\documentclass{aptpub}
\usepackage{graphicx}
\usepackage{multicol}
\usepackage{caption}
\usepackage{subcaption}
\usepackage{xcolor}

\authornames{Timofei Prasolov, Sergey Foss, Seva Shneer} 
\shorttitle{Limit theorems and structural properties of the Cat-and-Mouse Markov chain} 

\newcommand{\eq}[1]{(\ref{eq:#1})}
\newcommand{\lemt}[1]{Lemma~\ref{lem:#1}}
\newcommand{\thr}[1]{Theorem~\ref{thr:#1}}
\newcommand{\fig}[1]{Figure~\ref{fig:#1}}
\newcommand{\app}[1]{Appendix~\ref{app:#1}}
\newcommand{\sectn}[1]{Section~\ref{sect:#1}}

\newcommand{\sect}[1]{\ref{sect:#1}}

\newcommand{\PB}{\mathbb{P}}
\newcommand{\Expect}{\mathbb{E}}
\newcommand{\D}{\mathbf{Var}}

\begin{document}

\title{Limit theorems and structural properties of the Cat-and-Mouse Markov chain and its generalisations} 

\authorone[Novosibirsk State University, Novosibirsk]{Timofei Prasolov} 
\authortwo[Heriot-Watt University, Edinburgh and MCA, Novosibirsk State University and Sobolev Institute of Mathematics, Novosibirsk]{Sergey Foss}
\authorthree[Heriot-Watt University, Edinburgh and MCA, Novosibirsk State University, Novosibirsk]{Seva Shneer}

\addressone{Pirogova 1, Novosibirsk State University, Novosibirsk, 630090\\ \email{prasolov.tv@yandex.ru}\\
Research was financially s supported by the Mathematical Center in Akademgorodok under Agreement No. 075-15-2019-1613 with the Ministry of Science and Higher Education of the Russian Federation} 
\addresstwo{Heriot-Watt University, Edinburgh, EH14 4AS\\ \email{sergueiorfoss25@gmail.com}} 
\addressthree{Heriot-Watt University, Edinburgh, EH14 4AS\\ \email{v.shneer@hw.ac.uk}} 

\begin{abstract}
We revisit the so-called Cat-and-Mouse Markov chain, studied earlier by Litvak and Robert (2012). This is a 2-dimensional Markov chain on the lattice $\mathbb{Z}^2$, where the first component (the cat) is a simple random walk and the second component (the mouse) changes when the components meet. We obtain new results for two generalisations of the model. Firstly, in the 2-dimensional case we consider far more general jump distributions for the components and obtain a scaling limit for the second component. When we let the first component be a simple random walk again, we further generalise the jump distribution of the second component. Secondly, we consider chains of three and more dimensions, where we investigate structural properties of the model and find a limiting law for the last component.   
\end{abstract}

\keywords{cat-and-mouse games; multidimensional Markov chain; compound renewal process; regular variation; weak convergence; randomly stopped sums.} 

\ams{60J10}{60F05} 

\section{Introduction}\label{sect:introduction} 


We analyse the dynamics of a stochastic process with dependent coordinates, commonly referred to as the Cat-and-Mouse (CM) Markov chain (MC), and of its generalisations. Let $\mathcal{S}$ be a directed graph. Let $\{(C_n, M_n)\}_{n=0}^\infty$ denote the CM MC on $\mathcal{S}^2$, defined as follows. 
We assume that $\{C_n\}_{n=0}^\infty$, the location of the cat, is a MC on $\mathcal{S}$ with transition matrix $P = (p(x, y)),x, y \in \mathcal{S})$. The second coordinate, the location of the mouse, $\{M_n\}_{n=0}^\infty$ has the following dynamics:
\begin{itemize}
\item If $M_n \neq C_n$, then $M_{n+1} = M_n$,
\item If $M_n = C_n$, then, conditionally on $M_n$, the random variable $M_{n+1}$ has distribution $(p(M_n, y), y \in \mathcal{S})$ and is independent of $C_{n+1}$.
\end{itemize}
In our model the cat is trying to catch the mouse. The mouse is usually in hiding and not moving but, if the cat hits the same location of the graph, the mouse jumps. The cat does not notice where the mouse jumps to, so it proceeds independently.

CM MC is an example of models called Cat-and-Mouse games. CM games are common in game theory. We refer to Coppersmith \textit{et al.} (1993), where the authors showed that a CM game is at the core of many on-line algorithms and, in particular, may be used in settings considered by Manasse \textit{et al.} (1990) and Borodin \textit{et al.} (1992). Some special cases of CM games on the plane have been studied by Baeza-Yates \textit{et al.} (1993). Two examples of CM games have been discussed in Aldous and Fill (2002) in the context of reversible MCs.

There are many related models in applied probability where time evolution of the process may be represented as a multi-component Markov chain where one of the components has independent dynamics and forms a Markov chain itself (for example Gamarnik (2004), Gamarnik and Squillante (2005), Borst \textit{et al.} (2008), Foss \textit{et al.} (2012)). Typically such dependence is modelled using Markov modulation. In this paper we consider the case where the first component is a random walk. Thus, our model can be viewed as a random-walk-modulated random walk. We consider null-recurrent and transient cases where we find proper scaling for the components.

We are mainly motivated by the results of the paper by Litvak and Robert (2012) where the authors analyse scaling properties of the (non-Markovian) sequence $\{M_n\}_{n=0}^\infty$ for a specific transition matrix $P$ when $\mathcal{S}$ is either $\mathbb{Z}, \mathbb{Z}^2$ or $\mathbb{Z}^+$. Although it deals with a relatively simple case of the CM MC, it clearly illustrates a certain phenomenon related to this type of Markov modulation. We will focus on the case $\mathcal{S} = \mathbb{Z}$. In addition, assume that the transition matrix $P$ satisfies $p(x, x+1) = p(x, x-1) = 1/2$. 
It was proven in Theorem 3 of Litvak and Robert (2012) that the convergence in distribution
\begin{align*}
\left\{\frac{1}{\sqrt[4]{n}}M_{[nt]}, t\geq 0\right\} \Rightarrow \left\{B_1 (L_{B_2}(t)), t\geq 0\right\}, \  \  
\mbox{as} \ \  n\to\infty,  
\end{align*}
holds, where $B_1(t)$  and $B_2(t)$ are independent standard Brownian motions on $\mathbb{R}$ and $L_{B_2}(t)$ is the local time process of $B_2(t)$ at $0$. 
By convergence in distribution of c\`{a}dl\`{a}g
stochastic processes we mean convergence in the ${\mathcal J}_1$-topology
 (see Appendix A for further explanation).

This result looks natural, since the mouse, observed at the meeting times with the cat, is a simple random walk. The time intervals between meeting times are independent and identically distributed. They have the same distribution as the time needed for the cat (also a simple random walk) to get from $1$ to $0$, which has a regularly varying tail with parameter $1/2$ (see, e.g., Spitzer (1964)). Thus, the scaling for the location of the mouse is $\sqrt[4]{n} = (n^{1/2})^{1/2}$.  Local time $L_{B_2}(t)$ can be interpreted as the scaled duration of time the cat and the mouse spend together. 

This led us to question to what extent such a phenomenon holds, for what kind of distributions of jumps, and for which types of Markov modulation. In this paper we show that similar behaviour holds when jump-size distributions of both components have zero mean and finite variance. The behaviour slightly changes when we introduce heavier tails for the jump-size distribution of the mouse. For this case we develop a more general approach based on the work of Jurlewicz \textit{et al.} (2012). In parallel, we introduce additional components while applying an analogue of the aforementioned Markov modulation. Here, through the analysis of dynamical structural properties, we demonstrate a similar phenomenon for additional components.   

More specifically, we provide two generalisations of the CM MC introduced above. The first generalisation relates to the jump distribution of the mouse. Given $C_{n} = M_{n}$, random variable $M_{n+1} - M_n$ has a general distribution which has a finite first moment and belongs to the strict domain of attraction of a stable law with index $\alpha \in [1, 2]$, with a normalising function $\{b(n)\}_{n=1}^\infty$ (note that distributions with a finite second moment belong to the domain of attraction of a normal distribution) and a centering function $n\Expect(M_{n+1} - M_n)$. We find a weak limit of ${\{b^{-1}(\sqrt{n}) M_{[nt + 1]}\}_{t\geq 0}}$, as $n\to\infty$. This model is more challenging than the classical setting because, when the mouse jumps, the value of this jump and the time until the next jump may be dependent. Also, if the jump distribution of the mouse has an infinite second moment, we cannot use classical results for regenerative jump processes such as Theorem 5.1 from Kasahara (1984). Next, we consider the case where both components have general distributions with finite second moments. Here our results take into account the approach developed by Uchiyama (2011a).

In the second generalisation we add more components (we will refer to the objects whose dynamics these components describe as ``agents'') to the system, with keeping the chain ``hierarchy''. For instance, adding one extra agent (we refer to it as the dog), acting on the cat the same way as the cat acts on the mouse, slows down the cat and, therefore, also the mouse. We are interested in the effect of this on the scaling properties of the process. Recursive addition of further agents will slow down the mouse further. For the system with three agents we investigate the dynamical structural properties and find a weak limit of $\{n^{-1/8} M_{[nt]}\}_{t\geq 0}$, as $n\to\infty$. The system regenerates when all the agents are at the same point. Therefore, if we find the tail asymptotics of the time intervals between these events, we can split the process into i.i.d. cycles and use a limit theorem for regenerative processes by Kasahara (1984)).

We also consider systems with an arbitrary finite number of agents where we provide a relatively simple result on weak convergence, for a fixed $t>0$. The path analysis in this case turned out to be difficult and we only have partial results. Namely, we have studied the relation between the numbers of jumps of different agents and obtained the limiting law for the last agent. We have not succeeded in finding the asymptotics for the time intervals between regeneration points, which remains an open question.

We note that our paper is devoted to limit theorems for above-mentioned Markov modulated Markov chains. It is worth noting other interesting questions for such type of models not covered here: stability problems were studied in, e.g., [\ref{ShahShin}]--[\ref{FossShneerThomasWorral}] and large deviations problems in, e.g., [\ref{AsmussenHenriksenKloppelberg}]--[\ref{FossKonstantopoulosZachary}].

The paper is structured as follows. In \sectn{models&results} we define our models and formulate our results. In Sections \sect{CatandMouseModel} and \sect{trajectories} we analyse the trajectories of the CM MC and \textit{Dog-and-Cat-and-Mouse} (DCM) MC respectively. This analysis gives the main idea of the proof of our result on scaling properties of DCM MC (\thr{DogCatMouse}). In \sectn{generalised mouse}, we prove our results on scaling properties of general CM MC (\thr{GeneralMouse} and \thr{general cat and mouse}). We shift the time of our process and use characteristic functions to show that the conditions of Theorem 3.1 from Jurlewicz \textit{et al.} (2012) hold. In \sectn{dog cat mouse}, we prove \thr{DogCatMouse}. We approximate the dynamics of the mouse by considering only the values of the process at times when all agents are at the same point of the integer line and then use Theorem 5.1 from Kasahara (1984) to obtain the result. In the \sectn{linear food chain}, we prove our result on scaling properties for the system with an arbitrary finite number $N$ of agents. We approximate the $N$-th component of the system $X^{(N)}$ (that describes behaviour of the $N$'th agent) by the component $X^{(N-1)}$, slowed down by an independent renewal process.

The Appendix includes definitions and proofs of supplementary results. In \app{m_1-topology}, we define weak convergence of stochastic processes.  \app{NewProp3} clarifies the asymptotic closeness of two scaled processes related to Theorem 1. Finally, in \app{RandStopSumAsymp} we provide auxiliary results on randomly stopped sums with positive summands having a regularly varying tail distribution and infinite mean. 

Throughout the paper we use the following conventions and notation. For two ultimately positive functions $f(t)$ and $g(t)$ we write $f(t) \sim g(t)$  if $\lim_{t\to\infty} f(t)/g(t) = 1$. For any event $A$, its \textit{indicator function} $I[A]$ is a random variable that takes value $1$ if the event occurs, and value $0$, otherwise. 
We use symbol ${\Rightarrow}$ for the weak convergence of distributions of random variables or vectors, 
and $\overset{\mathcal{D}}{\Rightarrow}$ for convergence of trajectories of random processes in ${\mathcal J}_1$-topology (see Appendix A). Next, $X\overset{d}{=}Y$
means that random variables $X$ and $Y$ are identically distributed.  Finally we use the following abbreviations: CM -- Cat-and-Mouse, DCM -- Dog-and-Cat-and-Mouse, MC -- Markov chain, i.i.d. -- independent and identically distributed, r.v. -- random variable, a.s. -- almost surely, w.p. -- with probability.

\section{Models and results}\label{sect:models&results}
In this section we recall the CM MC on the integers and introduce several of its generalisations.
\subsection{``Standard''\ Cat-and-Mouse Markov chain on $\mathbb{Z}$ \ $(C\to M)$}
Let $\xi = \pm 1$ w.p. $1/2$. Let $\{\xi^{(1)}_n\}_{n=1}^\infty$ and $\{\xi^{(2)}_n\}_{n=1}^\infty$ be two mutually independent sequences of independent copies of $\xi$. Given $C_0 = M_0 = 0$, we define the dynamics of CM MC $(C_n, M_n)$ as follows:
$$C_{n} = C_{n-1} + \xi^{(1)}_{n},$$
$$M_n = M_{n-1} + \begin{cases}
0, & \text{if $C_{n-1} \neq M_{n-1},$}\\
\xi^{(2)}_n, & \text{if $C_{n-1} = M_{n-1},$}
\end{cases}$$
for $n \geq 1$.

Let $D[[0,\infty), \mathbb{R}]$ denote the space of all right-continuous functions on $[0, \infty)$ having left limits (RCLL, or c\`{a}dl\`{a}g functions).

Let $M(nt) = M_{[nt]}$, $t\geq0$, be a continuous-time stochastic process taking values $M_k$, $k\geq 0$, for $t\in[k/n, (k+1)/n)$. Clearly, it is piecewise constant and its trajectories belong to $D[[0,\infty), \mathbb{R}]$.

Litvak and Robert (2012) have proved convergence of trajectories 
\begin{equation*}\label{eq:LitvakRobert}
\left\{\frac{1}{\sqrt[4]{n}}M(nt), t\geq 0\right\} \overset{\mathcal{D}}{\Rightarrow} \left\{B_1 (L_{B_2}(t)), t\geq 0\right\}, \ \text{as $n\to\infty$}, 
\end{equation*}
 where $B_1(t)$  and $B_2(t)$ are independent standard Brownian motions on $\mathbb{R}$ and $L_{B_2}(t)$ is the local time process of $B_2(t)$ at $0$. 

\subsection{Cat-and-Mouse model with a general jump distribution of the mouse \ $(C \to M)$}
In this Subsection we introduce our results for CM MC with more general distributions of random variables $\xi^{(1)}_n$ and $\xi^{(2)}_n$. We start with the same distribution of $\xi^{(1)}_n$ and generalise distribution of $\xi^{(2)}_n$. Thus, the cat is a simple random walk and  the mouse is a general random walk. We then proceed to generalise the distribution of $\xi^{(1)}_n$ as well.

\textbf{2.2.1} We continue to assume that the dynamics of the cat is described by a simple random walk on $\mathbb{Z}$.  Let $\xi = \pm 1$ w.p. $1/2$. Let $C_0 = 0, \ C_{n} = C_{n-1} + \xi^{(1)}_n$, where $\xi, \xi^{(1)}_1, \xi^{(1)}_2, \ldots$ are independent and identically distributed (i.i.d) r.v.'s.

Let $M_0 = 0, \ M_n = M_{n-1} + \xi^{(2)}_n I[C_{n-1} = M_{n-1}]$ where $\{\xi^{(2)}_n\}_{n=1}^\infty$ are i.i.d r.v.'s independent of $\{\xi^{(1)}_n\}_{n=1}^\infty$. Assume  that 
\begin{equation}\label{eq:finite mean}
\mu = \Expect \xi^{(2)}_1 \ \text{is finite}
\end{equation}
and that the distribution of $\xi_1^{(2)}$ belongs to the domain of attraction of a strictly stable distribution, i.e. 
there exist a function $b(c) >0$, $c\geq0$, and a r.v. $A^{(2)}$ having a strictly stable distribution with index $\alpha \in [1,2]$ such that the weak convergence 
\begin{equation}\label{eq:stable law A}
\frac{\sum_{k=1}^{n} (\xi^{(2)}_k - \mu)}{b(n)} \Rightarrow A^{(2)}, \ \text{as $n\to\infty$}
\end{equation}
holds. Define
\begin{align*}
\tau(0) = 0 \ \text{and} \ \tau(n) = \inf\{m > \tau(n-1): \ C_m = M_m\}, \ \text{for $n\ge 1$}.
\end{align*}
Given \eq{finite mean}, we show that the tail-distribution of $\tau(1)$ is regularly varying with index $1/2$.  It follows then that there exists a r.v. $D^{(2)}$ having a strictly stable distribution with index $1/2$ such that
\begin{equation}\label{eq:T^2}
\frac{\tau(n)}{n^2} \Rightarrow D^{(2)}, \ \text{as $n\to\infty$.}
\end{equation}

In the proof of \thr{GeneralMouse} we show that, in fact, there is a weak convergence of two-dimensional random vectors 
\begin{align*}
\left(\frac{\sum_{k=1}^{n} (\xi^{(2)}_k - \mu)}{b(n)}, \  \frac{\tau(n)}{n^2}\right) \ \Rightarrow (A^{(2)}, D^{(2)}), \ \text{as $n\to\infty$,}
\end{align*}
where the r.v.'s on the right-hand side are independent. Further, let $\{(A^{(2)}(t), D^{(2)}(t))\}_{t\geq 0}$ denote a stochastic process with independent increments such that $(A^{(2)}(1), D^{(2)}(1))$ has the same distribution as $(A^{(2)}, D^{(2)})$, or, equivalently, the \textit{L\'evy process} generated by $(A^{(2)}, D^{(2)})$. Let $E^{(2)}(s) = \inf\{t\geq0 : \ D^{(2)}(t)> s\}$, which is a multiple
of $L_{B_2}(s)$, the local time at zero of a standard Brownian motion.  Thus, 
the following result is a natural extension of the result by Litvak and Robert (2012), see Remark below.

\begin{theorem}\label{thr:GeneralMouse}
Assume that \eq{finite mean} and \eq{stable law A} hold. Then
\begin{itemize}
\item if $\mu = 0$, we have \begin{equation}\label{Thr_general_mouse_without_drift}\
\left\lbrace \frac{M(nt + 1)}{b(\sqrt{n})}, \ t\geq 0\right\rbrace \overset{\mathcal{D}}{\Rightarrow} \{A^{(2)}(E^{(2)}(t)), t\geq 0\}, \ \text{as $n\to\infty$,}
\end{equation}
\item if $\mu \neq 0$, we have \begin{equation}\label{Thr_general_mouse_with_drift}\
\left\lbrace \frac{M(nt + 1)}{\sqrt{n}}, \ t\geq 0\right\rbrace \overset{\mathcal{D}}{\Rightarrow} \{\mu E^{(2)}(t), t\geq 0\}, \ \text{as $n\to\infty$.}
\end{equation}
\end{itemize} 
\end{theorem}

\textit{Remark} 
If r.v.'s $\xi_n^{(2)}$ are bounded, then the function $b(n)$ is proportional to $\sqrt{n}$, and it is easy to show that
the processes $\frac{M(nt + 1)}{\sqrt[4]{n}}$  and $\frac{M(nt)}{\sqrt[4]{n}}$ are equivalent
(see Appendix B for details). Then the  results of Theorem 1 may be reformulated for the scaled process $M(nt)$ in place of $M(nt + 1)$ .

\textbf{2.2.2} Assume now that both $\xi^{(1)}_1$ and $\xi^{(2)}_1$ have general distributions on the integer lattice. The main difference for the mouse is that we need now to assume that the second moment of $\xi^{(2)}_1$ is finite. Our main result in this setting is that replacing the simple random walk with a general random walk does not change the scaling if we assume aperiodicity and finite second moments for the increments. By aperiodicity (strong aperiodicity in the sense of Spitzer (1964)) we mean here that
\begin{align*}
G.C.D. \{ k: \ {\mathbb P} (\xi_1^{(1)}=k)>0\} =1,
\end{align*} 
where ``G.C.D.'' stands for the ``greatest common divisor''.

\begin{theorem}\label{thr:general cat and mouse}
Assume that $\Expect \xi^{(1)} =0$, $0 <\D \xi^{(1)}_1 < \infty$ and $\xi^{(1)}_1$ has an aperiodic distribution. Assume $0< \D \xi^{(2)}_1 < \infty$ and, therefore, \eq{stable law A} holds with $b(n) = \sqrt{n\D \xi^{(2)}_1}$ and a standard normal r.v. $A^{(2)}$. Then the statements $(\ref{Thr_general_mouse_without_drift})$ and $(\ref{Thr_general_mouse_with_drift})$ of Theorem \textbf{\ref{thr:GeneralMouse}}  continue to hold, with $b(\sqrt{n}) = n^{1/4}\sqrt{\D \xi^{(2)}_1}$ in $(\ref{Thr_general_mouse_without_drift})$.
\end{theorem}

\subsection{Dog-and-Cat-and-Mouse model \ ($D \to C \to M$)}
In this Subsection we present a generalisation of the CM MC to the case of $3$ dimensions.
Let $\xi = \pm 1$ w.p. $1/2$. Let $\{\xi^{(1)}_n\}_{n=1}^\infty$, $\{\xi^{(2)}_n\}_{n=1}^\infty$ and $\{\xi^{(3)}_n\}_{n=1}^\infty$ be mutually independent sequences of independent copies of $\xi$. Given $D_0 = C_0 = M_0 = 0$, we can define the dynamics of DCM MC $\{(D_n, C_n, M_n)_n\}_{n=1}^\infty$ as follows: for $n\ge 1$, 
$$D_{n} = D_{n-1} + \xi^{(1)}_{n},$$
$$C_n = C_{n-1} + \begin{cases}
0, & \text{if $D_{n-1} \neq C_{n-1},$}\\
\xi^{(2)}_n, & \text{if $D_{n-1} = C_{n-1},$}
\end{cases}$$
$$M_n = M_{n-1} + \begin{cases}
0, & \text{if $C_{n-1} \neq M_{n-1},$}\\
\xi^{(3)}_n, & \text{if $C_{n-1} = M_{n-1}.$}
\end{cases}$$

Let $T^{(3)}(0) = 0$ and $T^{(3)}(k) = \min\{n >T^{(3)}(k-1): \ D_n = C_n = M_n\}$, for $k\geq 1$. We show that the tail-distribution of $T^{(3)}(1)$ is regularly varying with index $1/4$. Further, we show that there exists a positive r.v. $D^{(3)}$ with a stable distribution and Laplace transform $\exp(-s^{1/4})$ such that
\begin{equation}\label{eq:T^3}
\frac{T^{(3)}(k)}{2^7 k^4} \Rightarrow D^{(3)}, \ \text{as $k\to\infty$.}
\end{equation}

Let $\{D^{(3)}(t)\}_{t\geq 0}$ be a L\'evy process generated by $D^{(3)}$ and $E^{(3)}(s) = \inf\{t\geq0 : \ D^{(3)}(t)> s\}$.

\begin{theorem}\label{thr:DogCatMouse}
We have $\Expect  M(T^{(3)}(1)) = 0$, $\sigma^2 =  \D M(T^{(3)}(1)) = 2$, and
\begin{align*}
\left\lbrace \frac{M(nt)}{2^{-7/8} n^{1/8}\sigma}, t\geq 0\right\rbrace \overset{\mathcal{D}}{\Rightarrow} \{B(E^{(3)}(t)), t\geq 0\},\ \text{as $n\to\infty$, }
\end{align*}
where $B(t)$ is a standard Brownian motion, independent of $E^{(3)}(t)$.
\end{theorem}

\subsection{Linear hierarchical chains $(X^{(1)} \to X^{(2)} \to \ldots \to X^{(N)})$ of length $N$}
In this Subsection we consider a generalisation of the CM MC to the case of $N$ dimensions. Due to the complexity of sample paths for $N>3$, it does not seem to be possible to prove an analogue of \eq{T^2} and \eq{T^3}. For this setting, we prove the convergence for every fixed $t>0$.

Let $\xi = \pm 1$ w.p. $1/2$. Let $\{\{\xi^{(j)}_n\}_{n=1}^\infty\}_{j=1}^N$ be mutually independent sequences of independent copies of $\xi$. Assume $X^{(1)}_0 = \ldots = X^{(N)}_0 = 0$. Then MC $(X^{(1)}_n, \ldots, X^{(N)}_n)$ is defined as follows:
$$X^{(1)}_{n} = X^{(1)}_{n-1} + \xi^{(1)}_{n},$$
$$X^{(j)}_n = X^{(j)}_{n-1} + \begin{cases}
0, & \text{if $X^{(j-1)}_{n-1} \neq X^{(j)}_{n-1},$}\\
\xi^{(j)}_n, & \text{if $X^{(j-1)}_{n-1} = X^{(j)}_{n-1},$}
\end{cases}$$
for $j \in \{2, \ldots, N\}$ and for $n \geq 1$.

For the result below we need the following distribution. Let $G_\alpha$ be the distribution function of a one-sided strictly stable law satisfying the condition $x^\alpha (1- G_\alpha(x)) \to (2-\alpha) / \alpha$, as $x\to\infty$.
\begin{theorem}\label{thr:LinearFoodChain}
Let $\{\zeta_i\}_{i=1}^\infty$ be i.i.d. r.v.'s satisfying $\PB\{\zeta_i \geq y\} = G_{1/2}(9/y^2)$. Let $\psi$ be n r.v. with a standard normal distribution and independent of $\{\zeta_i\}_{i=1}^\infty$. Then, for any fixed $t>0$, we have
$$\frac{X^{(N)}_{[nt]}}{n^{1/2^N}} \Rightarrow t^{1/2^N} \psi \prod_{i=1}^{N} \sqrt{\sqrt{\frac{\pi}{2}}\zeta_i}, \ \text{as $n\to\infty$.}$$
\end{theorem}

\section{Trajectories of the ``standard'' Cat-and-Mouse model}\label{sect:CatandMouseModel}

In order to prove Theorems 1-4, we first revisit the ``standard'' CM model and highlight a number of properties that are of use in the analysis of more general CM and DCM models discussed here.

We assume that $C_0 = M_0 = 0$. Let $V_n = |C_n - M_n|$, for $n\geq 0$. We can write $M_{n+1} = M_{n} + \xi^{(2)}_{n+1}I[V_{n} = 0]$, for $n\geq 1$. Then $V_{n+1} = |C_{n+1}-M_{n+1}| = |C_n - M_n + \xi^{(1)}_{n+1} - \xi^{(2)}_{n+1}I[V_{n} = 0]|$. We can further observe that
\begin{equation*}
V_{n+1} =  \begin{cases}
|\xi^{(1)}_{n+1} - \xi^{(2)}_{n+1}|\overset{d}{=} 1 + \xi^{(1)}_{n+1}, \ \text{if $V_n = 0$,} \\ 
|C_n - M_n + \xi^{(1)}_{n+1}| \overset{d}{=} V_n + \xi^{(1)}_{n+1}, \ \text{if $V_n \neq 0$}
\end{cases}
\end{equation*}
Thus, $V_n$ forms a MC. Let $p_i(j) = \PB\{V_{n+1}=j | V_n = i\}$, for $i,j\geq 0$. Note that $p_0(j) = p_1(j)$ for any $j$. 

Let 
\begin{align*}
U^{(2)}(0) = 0\ \text{and} \ U^{(2)}(k) = \min\{n> U^{(2)}(k-1):\ V_n \in \{0, 1\}\}.
\end{align*}
Since $p_0(j) = p_1(j)$ for any $j$, we have that r.v.'s $\{U^{(2)}(k) - U^{(2)}(k-1)\}_{k=1}^\infty$ are i.i.d. and r.v. $(U^{(2)}(k) - U^{(2)}(k-1))$ does not depend on $V_{U^{(2)}(k-1)}$, for $k\ge 1$. From the Markov property we have
\begin{align*}
V_{U^{(2)}(k) + 1} \overset{d}{=} 1 + \xi^{(1)}_{1} = \begin{cases}
0, & \text{w.p. $\frac{1}{2}$,}\\
2, & \text{w.p. $\frac{1}{2}$.}
\end{cases}
\end{align*}

Thus, after each time-instant $U^{(2)}(k)$ the cat and the mouse jump with equal probabilities either to the same point or to two different points distant by $2$. In the latter case, $V_{U^{(2)}(k+1)} = 1$, since the cat's jumps are $1$ or $-1$. For the cat, let $\tau^{(1)}_m = \min \{n: \sum_{k=1}^n \xi^{(1)}_k = m\}$ denote the hitting time of the state $m$. Then
\begin{align*}
U^{(2)}(1) \overset{d}{=} 1 + \begin{cases}
0, & \text{w.p. $\frac{1}{2}$,}\\
\tau^{(1)}_1, & \text{w.p. $\frac{1}{2}$.}
\end{cases}
\end{align*} 

The tail asymptotics for $\tau^{(1)}_1$ are known: $\PB\{\tau^{(1)}_1>n\} \sim \sqrt{2/(\pi n)}$, as $n\to \infty$ (see, e.g., Section III.2 in Feller (1971a) for related result). Since $\tau^{(1)}_1$ has a distribution with a regularly varying tail, for any $m=2, 3, \ldots$ we have
\begin{align*}
\PB\{\tau^{(1)}_m > n\} \sim m\PB\{\tau^{(1)}_1 > n\} \sim \sqrt{2m^2/(\pi n)}, \ \text{as $n\to\infty$.}
\end{align*}

\section{Trajectories in the Dog-and-Cat-and-Mouse model}\label{sect:trajectories}

In this Section we study structural properties of the DCM MC on $\mathbb{Z}$. We describe the main idea of the analysis which may be of independent interest as, we believe, it may be applied to other multi-component MCs.

As before, let $\{T^{(3)}(n)\}_{n=0}^\infty$ be the meeting time-instants, when all three agents meet at a certain point of $\mathbb{Z}$, and let $\{J_k^{(3)}\}_{k=1}^\infty = \{T^{(3)}(k) - T^{(3)}(k-1)\}_{k=1}^\infty$ be the times between such events. Let $M_{T^{(3)}(n)}$, $n=0, 1, \ldots$, be the locations of the mouse (and, therefore, the common locations of the agents) at the embedded epochs $T^{(3)}(n)$ and $\{Y_k^{(3)}\}_{k=1}^\infty = \{M_{T^{(3)}(k)} - M_{T^{(3)}(k-1)}\}_{k=1}^\infty$ the corresponding jump sizes between the embedded epochs. Due to time homogeneity, random vectors $\{(Y_k^{(3)}, J_k^{(3)})\}$ are i.i.d..

Let $N(t) = \max\{n: \ T^{(3)}(n) = \sum_{k=1}^n J_k^{(3)} \leq t\}$, for $t\geq0$. Let $S_0 = 0$ and $S_n = \sum_{k=1}^n Y_k^{(3)}$. We show that the statement of \thr{DogCatMouse} holds if we replace  $M_n$ with a continuous-time process
\begin{equation*}\label{M_tilde}
\widetilde{M}(t) = S_{N(t)} = \sum_{k=1}^{N(t)} Y_k, \ \text{for $t\geq 0$}.
\end{equation*} 

The process $\widetilde{M}(t)$ is a so-called \textit{coupled continuous-time random walk} (see Becker-Kern \textit{et al.} (2004)), and we use Theorem 5.1 from Kasahara (1984) to obtain its scaling properties.

\subsection{Distribution of r.v. $J^{(3)}_1$}\label{sect:Section Distribution of r.v. J}

We assume that $D_0 = C_0 = M_0 = 0$. Let $V_n = (V_{n1}, V_{n2}) = (|D_n - C_n|, |C_n - M_n|)$. Then the following recursion holds: 
$$(D_{n+1}, C_{n+1}, M_{n+1}) = (D_n + \xi^{(1)}_{n+1}, C_n + \xi^{(2)}_{n+1} I[V_{n1} = 0],  M_n + \xi^{(3)}_{n+1} I[V_{n2} = 0]).$$
Note further that
\begin{equation}\label{eq:00022022}
V_{n+1} \overset{d}{=} \begin{cases}
(1 + \xi^{(1)}_{n+1}, 1+\xi^{(2)}_{n+1}), \ \text{if $V_{n1} = V_{n2} = 0$,} \\ 
(1 + \xi^{(1)}_{n+1}, V_{n2} + \xi^{(2)}_{n+1}), \ \text{if $V_{n1} = 0$ and $V_{n2} \neq 0$,}\\
(V_{n1} + \xi^{(1)}_{n+1}, 1), \ \text{if $V_{n1} \neq 0$ and $V_{n2} = 0$,}\\
(V_{n1} + \xi^{(1)}_{n+1}, V_{n2}), \ \text{if $V_{n1} \neq 0$ and $V_{n2} \neq 0$.}
\end{cases}
\end{equation}

Thus, $\{V_n\}_{n=0}^\infty$ is a MC. Let $p_{ij}(k,l) = \PB\{V_{n+1}=(k, l) | V_n = (i, j)\}$, for $i,j, k,l\geq 0$. Note that $p_{00}(k,l) = p_{01}(k,l)$ for any $k,l \ge 0$.

Let 
\begin{align*}
U^{(3)}(0) = 0\ \text{and} \ U^{(3)}(k) = \min\{n> U^{(3)}(k-1):\ V_n \in  \{(0,0), (0,1)\}\}.
\end{align*}
Since $p_{00}(m,l) = p_{01}(m,l)$ for any $m,l$, we have that r.v.'s
 $\{U^{(3)}(k) - U^{(3)}(k-1)\}_{k=1}^\infty$ are i.i.d. and r.v. $(U^{(3)}(k) - U^{(3)}(k-1))$ does not depend on $V_{U^{(3)}(k-1)}$, for $k\ge 1$.

In other words, each time $n=U^{(3)}(k)$ the DCM process visits either a state on the ``main diagonal'' $D_n=C_n=M_n$ or an  auxiliary state $D_n = C_n = M_n \pm 1$. To find the tail asymptotics
for $T^{(3)}_1$,  we find tail asymptotics for $U^{(3)}(1)$ and then use the natural link between time-instants $T^{(3)}(1)$ and $\{U^{(3)}(k)\}_{k=1}^\infty$.

\begin{lemma}\label{lem:DCM trajectory}
Let $V_0 \in \{(0,0), (0,1)\}$.  Then we have
$$\PB\{U^{(3)}(1) > n\} \sim \frac{1}{2^{1/4}\Gamma(3/4) n^{1/4}}, \ \text{as $n\to\infty$.}$$ Further, $U^{(3)}(1) = 1$ iff $V_{U^{(3)}(1)} = (0,0)$.
\end{lemma}
\begin{proof}
Let $V_0 = (0,0)$. It is apparent from the first line of equation \eq{00022022}  that
\begin{align*}
\PB\{V_1 = (0,0)\} = \PB\{V_1 = (2,0)\} = \PB\{V_1 = (0,2)\} = \PB\{V_1 = (2,2)\} = \frac{1}{4}.
\end{align*}
Since $p_{00}(k,l) = p_{01}(k,l)$,  r.v. $V_1$ has the same distribution given $V_0 = (0, 1)$.

\begin{figure}[!htp]
  \begin{subfigure}[b]{0.23\textwidth}
    \includegraphics[width=\textwidth]{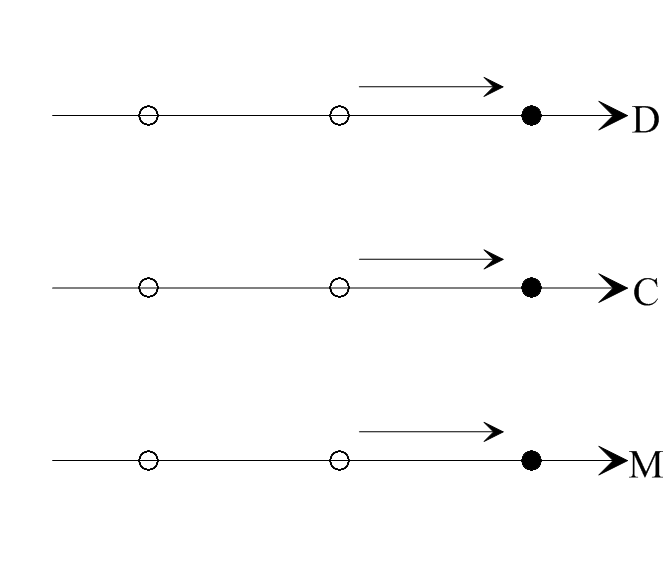}
    \caption{$V_1 = (0,0)$.}
    \label{fig:alpha0000}
  \end{subfigure}
  \hfill
  \begin{subfigure}[b]{0.23\textwidth}
    \includegraphics[width=\textwidth]{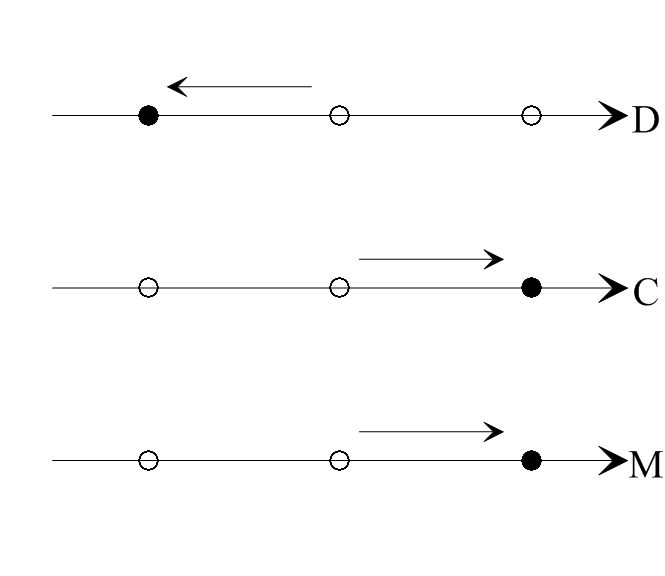}
    \caption{$V_1 = (2,0)$.}
    \label{fig:alpha0020}
  \end{subfigure}
  \hfill
  \begin{subfigure}[b]{0.23\textwidth}
    \includegraphics[width=\textwidth]{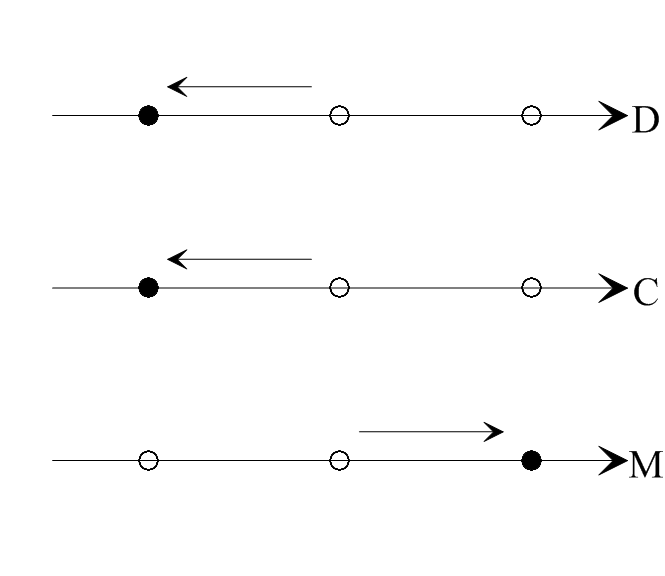}
    \caption{$V_1 = (0,2)$.}
    \label{fig:alpha0002}
  \end{subfigure}
  \hfill
  \begin{subfigure}[b]{0.23\textwidth}
    \includegraphics[width=\textwidth]{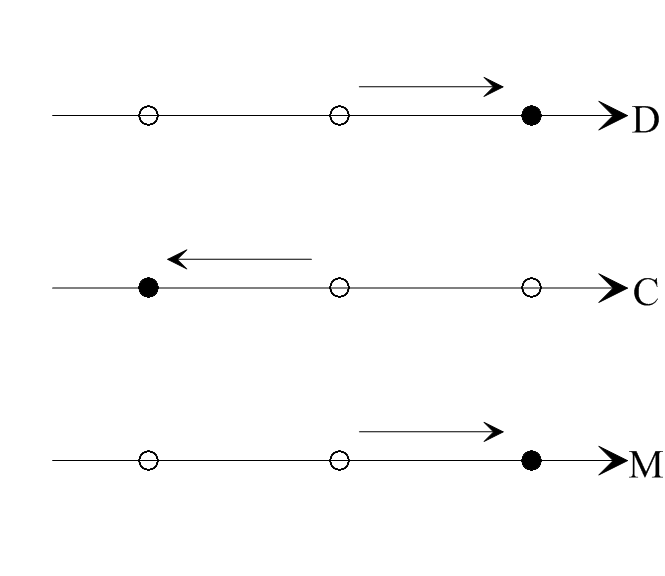}
    \caption{$V_1 = (2,2)$.}
    \label{fig:alpha0022}
  \end{subfigure}
  \caption{The positioning after the first jump.}
\end{figure}

Let $V_1 = (0, 2)$ (\fig{alpha0002}). From the second and the fourth lines of equation \eq{00022022} we know that $|V_{(k+1)2} - V_{k2}| \in \{0, 1\}$, for $k\ge 1$, given $V_{k2} \neq 0$. Therefore $V_{k2}$ arrives at $1$ before hitting $0$ and $V_{U^{(3)}(1)} = (0, 1)$. Let $\tau, \tau_1, \tau_2, \ldots$ be independent copies of $\tau^{(1)}_1$. Then $U^{(3)}(1)$ has the same distribution as $\sum_{k=1}^\tau \tau_k$ and we have that 
$$\PB\{\sum_{k=1}^\tau \tau_k > n\} \sim n^{-1/4} \frac{\Gamma^{1/2}(1/2) \Gamma(1/2)}{\Gamma(3/4)} \sqrt{\sqrt{\frac{2}{\pi}}} \sqrt{\frac{2}{\pi}} = \frac{2^{3/4}}{\Gamma(3/4) n^{1/4}},$$ as $n\to\infty$ (see Appendix C).

Let $V_1 = (2,2)$ (\fig{alpha0022}). From the fourth line of equation \eq{00022022}, $V_{k2}$ remains at $2$ (the cat and the mouse do not move) until $V_{k1}$ reaches $0$. This happens after a time which has the same distribution as $\tau^{(1)}_2 = \min\{n > 0: \ \sum_{k=1}^n \xi^{(1)}_k = 2\}$. Thus, we travel from $(2,2)$ to $(0, 2)$ while never hitting $(0,0)$. We also know that the tail asymptotics of the travel time are $\PB\{\tau^{(1)}_2 > n\} \sim \sqrt{8/\pi n}$, as $n\to\infty$. Therefore, we travel from $(2,2)$ to $(0, 2)$ much faster than from $(0,2)$ to $(0, 1)$ and, given $V_1 = (2,2)$, we again have $\PB\{U^{(3)}(1) > n\} \sim \PB\{\sum_{k=1}^\tau \tau_k > n\}$, as $n\to\infty$.

Finally, let $V_1 = (2, 0)$ (\fig{alpha0020}). From the third line of equation \eq{00022022} we have $V_{2} \overset{d}{=} (2+ \xi^{(1)}_{2}, 1)$ and $V_{U^{(3)}(1)} = (0, 1)$, where, given $V_1 = (2, 0)$,  $\PB\{U^{(3)}(1)>n\} \sim \sqrt{8/\pi n}$, as $n\to\infty$.

Thus,
\begin{align*}
\PB\{U^{(3)}(1) > n\} \sim \frac{1}{2}\PB\{\sum_{k=1}^\tau \tau_k > n\} \sim \frac{1}{2^{1/4}\Gamma(3/4) n^{1/4}}, \ \text{as $n\to\infty$.}
\end{align*}
\end{proof}

We note the following relation between time-instants $T^{(3)}(1)$ and $\{U^{(3)}(k)\}_{k=1}^\infty$. Each time $n$ we are at the auxiliary state (when $V_n = (0, 1)$) we have a probability $1/4$ to jump into the state $D_{n+1} = C_{n+1} = M_{n+1}$ independently of anything else.  Using \lemt{DCM trajectory} and the results of Section XIII.6 from Feller (1971b) we get the following result.
\begin{proposition}\label{tail of J}
Let 
\begin{align*}
\nu = \inf\{k \ge 1: \ U^{(3)}(k) - U^{(3)}(k-1) = 1\}.
\end{align*} 
Then $\nu$ has a geometric distribution with parameter $1/4$, $T^{(3)}\equiv J^{(3)}=U^{(3)}(\nu)$ a.s.  and 
\begin{align*}
\PB\{J^{(3)}_1 > n\} = \PB\{U^{(3)}(\nu) > n\} \sim 4\PB\{U^{(3)}(1) > n\}, \ 
\text{as $n\to\infty$.}
\end{align*}
Therefore, there exists a positive r.v. $D^{(3)}$ having  a stable distribution with Laplace transform $\exp(-s^{1/4})$ such that
\begin{align*}
\frac{T^{(3)}(n)}{2^7 n^4} = \frac{\sum_{k=1}^n J^{(3)}_k}{2^7 n^4} \Rightarrow D^{(3)}, \ \text{as $n\to\infty$}.
\end{align*}
\end{proposition}

\subsection{Distribution of r.v. $Y^{(3)}_1$} 
In the previous Subsection we analysed the time our process spends between auxiliary states. In this Subsection we analyse  the total displacement of the mouse by time $T_1^{(3)}$. Note that the mouse may have either zero or one or two jumps between consecutive visits to auxiliary states and, therefore, the total number of jumps of the mouse by time
$T_1^{(3)}$, say $\varkappa$, admits the following simple upper bound:
\begin{align*}
\varkappa \le 2\nu \quad \mbox{a.s.}
\end{align*}

It follows that the sequence $Z_k = M_{U^{(3)}(k)} - C_{U^{(3)}(k)}$, $k=0,1,\ldots$ forms a time-homogeneous Markov chain and $Z_k \in \{-1, 0, 1\}$ a.s. Further, $Z_0 = Z_\nu = 0$ and, for any $k \in \{1, \nu - 1\}$, we have $Z_k = \pm 1$. 

Let 
\begin{align*}
\gamma^{(3)}_k = M_{U^{(3)}(k)} - M_{U^{(3)}(k-1)} \ \text{for $k\geq 1$.}
\end{align*}
Our analysis of trajectories of DCM MC in \sectn{Section Distribution of r.v. J} shows that $\gamma^{(3)}_k$ are conditionally independent given values of  the auxiliary Markov chain $\{Z_k\}_{k = 0}^\infty$. Clearly, 
$$Y^{(3)}_1 = M_{T^{(3)}(1)} = \sum_{k=1}^\nu \gamma^{(3)}_k.$$
We now find  the first and second moments of $Y_1^{(3)}$ and show that it has a light-tailed distribution.

Since $D_0 = C_0 = M_0 = 0$, we get $\nu=1$ if and only if $Z_1 = 0$ and $\gamma^{(3)}_1 = \pm 1$. Additionally, we have
\begin{align*}
\PB\{\gamma^{(3)}_1 = \pm 1,  \ Z_1 = 0 \ | \ Z_0 = 0\} = \PB\{D_1 = C_1 = M_1 = \pm 1\} = \frac{1}{8}.
\end{align*}
Another case of exactly one jump is when the cat and the mouse jump in different directions. Here we have
\begin{align*}
\PB\{\gamma^{(3)}_1 = \pm 1,  \ Z_1 = \pm 1 \ | \ Z_0 = 0\} = \PB\{C_1 = \mp 1, M_1 = \pm 1\} = \frac{1}{4}.
\end{align*}
Further, the mouse can have two jumps in the same direction w.p.
\begin{align*}
\PB\{\gamma^{(3)}_1 = \pm 2,  \ Z_1 = \pm 1 \ | \ Z_0 = 0\} = \PB\{D_1 = \mp 1, C_1 = M_1 = \pm 1, M_2 = \pm 2\} = \frac{1}{16},
\end{align*}
or two jumps in the opposite directions w.p.
\begin{align*}
\PB\{\gamma^{(3)}_1 = 0,  \ Z_1 = \pm 1 \ | \ Z_0 = 0\} = \PB\{D_1 = \pm 1, C_1 = M_1 = \mp 1, M_2 = 0\} = \frac{1}{16},
\end{align*}
Thus, given $Z_0 = 0$ we have
\begin{align*}
\PB\{Z_1 = 0\} = \frac{1}{4} \ \text{and} \ \PB\{Z_1 = \pm 1\} = \frac{3}{8},
\end{align*}
\begin{align}\label{eq:Z_0 = 0}
\PB\{\gamma^{(3)}_1 = 0\} = \frac{1}{8}, \ \PB\{\gamma^{(3)}_1 = \pm 1\} = \frac{3}{8} \ \text{and} \ \PB\{\gamma^{(3)}_1 = \pm 2\} = \frac{1}{16}.
\end{align}

From the above, we get that $\Expect \gamma^{(3)}_1 = 0$ and $\D \gamma^{(3)}_1 = 5/4$.

We now analyse the distribution of $\gamma^{(3)}_k$, $k\ge 2$. In order not to make the notation cumbersome, we assume that $D_0 = C_0 = 0$ and $M_0 = 1$, so $Z_0 = 1$. The case $Z_0 = -1$ is analogous by the symmetry. Then the mouse can make either zero jumps or exactly one jump before the next visit of the profess to an auxiliary state. In the case of zero jumps we have
\begin{align*}
\PB\{\gamma^{(3)}_1 = \pm 0,  \ Z_1 = 0 \ | \ Z_0 = 1\} = \PB\{D_1 = C_1 = 1\} = \frac{1}{4},
\end{align*}
\begin{align*}
\PB\{\gamma^{(3)}_1 = \pm 0,  \ Z_1 = 1 \ | \ Z_0 = 1\} = \PB\{C_1 = -1\} = \frac{1}{2}.
\end{align*}
In the case of exactly one jump we have
\begin{align*}
\PB\{\gamma^{(3)}_1 = 1,  \ Z_1 = 1 \ | \ Z_0 = 1\} = \PB\{D_1 =-1, C_1 = 1, M_2 =2 \} = \frac{1}{8},
\end{align*}
\begin{align*}
\PB\{\gamma^{(3)}_1 = -1,  \ Z_1 = -1 \ | \ Z_0 = 1\} = \PB\{D_1 =-1, C_1 = 1, M_2 =0\} = \frac{1}{8}.
\end{align*}
Thus, given $Z_0 = 1$ we have
\begin{align*}
\PB\{Z_1 = 0\} = \frac{1}{4}, \ \PB\{Z_1 = 1\} = \frac{5}{8} \ \text{and} \ \PB\{Z_1 = -1\} = \frac{1}{8}, 
\end{align*}
\begin{align}\label{eq:Z_0 = 1}
\PB\{\gamma^{(3)}_1 = 0\} = \frac{3}{4} \ \text{and} \ \PB\{\gamma^{(3)}_1 = \pm 1\} = \frac{1}{8}.
\end{align}
Thus, $\Expect \left( \gamma^{(3)}_1 | Z_0 = 1\right)  = 0$ and $\D \left( \gamma^{(3)}_1| Z_0 =1\right)  = 1/4$.

Let us return to the case $D_0 = C_0 = M_0 = 0$. Combining the results $\eqref{eq:Z_0 = 0}$ and $\eqref{eq:Z_0 = 1}$ we get
\begin{multline*}
\Expect Y^{(3)}_1 = \Expect \sum_{k=1}^\nu \gamma^{(3)}_k = \Expect \sum_{k = 1}^\infty \left( \gamma^{(3)}_k I[\nu \ge k]\right) = \Expect \gamma^{(3)}_1 + \sum_{k = 2}^\infty \Expect \left( \gamma^{(3)}_k I[\nu \ge k]\right)\\
 = \sum_{k = 2}^\infty \Expect \left( \gamma^{(3)}_k I[Z_{k-1} = 1] + \gamma^{(3)}_kI[Z_{k-1} = -1]\right) = 0.
\end{multline*}
In a similar manner we transform the second moment:
\begin{align}\label{eq:Second moment of Y_1}
\Expect (Y^{(3)}_1)^2 = \sum_{k=1}^\infty \Expect\left( ( \gamma^{(3)}_k)^2 I[\nu\ge k]\right) + 2\sum_{k=1}^\infty \sum_{m=k+1}^\infty \Expect \left(  \gamma^{(3)}_k \gamma^{(3)}_m I[\nu \ge m]\right). 
\end{align}
If we fix the values of $\{Z_k\}_{k=0}^\infty$ the r.v.'s $ \gamma^{(3)}_k$ and $ \gamma^{(3)}_m$ become independent. Combined with the fact that $\Expect \left( \gamma^{(3)}_m | Z_{m-1} = \pm 1 \right) = 0$, we get that the second sum in $\eqref{eq:Second moment of Y_1}$ equals zero. Now we can use the conditional second moments obtained above and the fact that the r.v. $\nu$ has a geometric distribution with parameter $1/4$. We obtain
\begin{align*}
\Expect (Y^{(3)}_1)^2 = \Expect (\gamma^{(3)}_1)^2 + \sum_{k=2}^\infty \Expect\left(  (\gamma^{(3)}_k)^2| \ Z_{k-1} = \pm 1\right)  \PB\{\nu\ge k\} = \frac{5}{4} + \frac{1}{4}(\Expect \nu -1) = 2.
\end{align*}
Finally, we observe that $|\gamma^{(3)}_1| \le 2$. Therefore,  $|Y^{(3)}_1| = |M_{J^{(3)}_1}| \le 2\nu$ a.s. and $|Y^{(3)|}$ has a finite exponential moment.  In particular, the following holds. 
\begin{proposition}\label{moments of Y_m}
We have $\Expect Y^{(3)}_1 = 0$, $\D Y^{(3)}_1 = 2$ and $\Expect |Y^{(3)}_1|^m < \infty$, for any $m\geq 3$.
\end{proposition}

\section{Proofs of main results}\label{sect:proofs}

In this section we provide proofs of our main results.

\subsection{Proofs of \thr{GeneralMouse} and \thr{general cat and mouse}}\label{sect:generalised mouse}

We start with the general idea of proofs of Theorems \ref{thr:GeneralMouse} and \ref{thr:general cat and mouse}.
For $i=1,2$, let $S^{(i)}_0 = 0$ and $S^{(i)}_n = \sum_{k=1}^n \xi^{(i)}_k$, for $n\geq1$. Let 
\begin{align*}
\tau(0) = 0 \ \text{and} \ \tau(n) = \inf\{m > \tau(n-1): \ S^{(1)}_m = S^{(2)}_n\}, \ \text{for $n\ge 1$}.
\end{align*}
Since $\{\xi^{(1)}_k\}_{k=1}^\infty$ and $\{\xi^{(2)}_k\}_{k=1}^\infty$ are independent sequences of i.i.d. r.v.'s, we have that $\tau(n) - \tau(n-1) \overset{d}{=} \tau(1)$, for $n \ge 1$. 

Let $\eta(t) = \max\{k\geq 0 :\ \tau(k) \leq t\}$ for $t\geq 0$. Define a continuous-time process $M'(t)$ by
\begin{align*}
M'(t) = 0 \ \text{for $t\in [0, 1)$ and} \ M'(t) = S^{(2)}_{\eta(t-1) + 1} =  \sum_{k=1}^{\eta(t-1)+1} \xi^{(2)}_k, \ \text{for $t\geq 1$.}
\end{align*}
It is straightforward to verify that $\{M'(n), \ n\ge 0\} \overset{d}{=} \{M(n), \ n\ge 0\}$. In the rest of the section we will omit the dash and simply write $M(t)$. The process $\{\widehat{M}(t)\}_{t\geq 0} = \{M'(t+1)\}_{t\geq 0}$ is a so-called \textit{oracle continuous-time random walk} (see, e.g., Jurlewicz \textit{et al.} (2012)).

First, consider the case $\Expect\xi^{(2)}_1 = 0$. We want to show that
\begin{equation}\label{eq:asymp_indep}
\left(\frac{S^{(2)}_n}{b(n)}, \frac{\tau(n)}{n^2} \right) \Rightarrow (A^{(2)}, D^{(2)}), \ \text{as $n\to\infty$.}
\end{equation} 
Given that, we will show that the first part of Theorem \textbf{\ref{thr:GeneralMouse}} follows from the next proposition. 
\begin{proposition}\label{Jurlewicz et al}
(Theorem 3.1, Jurlewicz \textit{et al.} (2012)). Assume \eq{asymp_indep} holds. Then 
\begin{align*}
\left\lbrace \frac{\widehat{M}(nt)}{b(\sqrt{n})}, t\geq 0 \right\rbrace = \left\lbrace \frac{M(nt + 1)}{b(\sqrt{n})}, t\geq 0 \right\rbrace \overset{\mathcal{D}}{\Rightarrow} \left\lbrace A^{(2)}(E^{(2)}(t)), t\geq 0 \right\rbrace, \ \text{as $n\to\infty$. }
\end{align*}
\end{proposition}
A similar result was proven in Theorem 3.6 from Henry and Straka (2011).

We will now show that relation \eq{asymp_indep} holds and that r.v.'s $A^{(2)}$ and $D^{(2)}$ are independent. We show that for certain functions $f_1$ and $f_2$,
\begin{multline}\label{eq:joint characteristic function}
\Expect \exp\left( i\left(\lambda_1\frac{S^{(2)}_n}{b(n)} + \lambda_2 \frac{\tau(n)}{n^2}\right) \right) = \Expect \exp\left( i\left(\lambda_1\frac{\xi^{(2)}_1}{b(n)} + \lambda_2 \frac{\tau(1)}{n^2}\right)\right)^n\\
 = \left(1 + \frac{f_1(\lambda_1) + f_2(\lambda_2)}{n} + o\left(\frac{1}{n} \right) \right)^n,
\end{multline}
as $n\to\infty$, for any $\lambda_1, \lambda_2 \in \mathbb{R}$. Indeed, convergence of characteristic functions is equivalent to weak convergence of r.v.'s and for independence of r.v.'s it is sufficient to verify that the characteristic function of the sum is equal to the product of respective characteristic functions. Since the right-hand side of \eq{joint characteristic function} converges to $\exp(f_1(\lambda_1))\exp(f_2(\lambda_2))$, it will prove the convergence and the independence of the limits $A^{(2)}$ and $D^{(2)}$.  

\textbf{5.1.1} We now present a proof of Theorem \textbf{\ref{thr:GeneralMouse}}. From condition \eq{stable law A} we have the weak convergence of $S_n^{(2)} / b(n)$  to the r.v. $A^{(2)}$. Again, this is equivalent to the convergence of characteristic functions. Thus,  \eq{stable law A} implies
\begin{align*}
\Expect\exp\left(i\lambda_1\frac{\sum_{k=1}^n\xi^{(2)}_k}{b(n)} \right) = \left[ \Expect\exp\left(i\lambda_1\frac{\xi^{(2)}_1}{b(n)} \right)\right]^n  \to \Expect\exp\left(i\lambda_1 A^{(2)}\right), \ \text{as $n\to\infty$.}
\end{align*}
Additionally, if $B^n(n) \to z$, as $n\to\infty$, then $n\log B(n) \to \log z$, which leads to $\log B(n) \sim n^{-1}\log z $ and hence, finally, $B(n) \sim 1 + n^{-1}\log z$, as $n\to\infty$. Thus, we have the following
\begin{equation}\label{eq:StepCharFunction}
\Expect\exp\left(i\lambda_1\frac{\xi^{(2)}_1}{b(n)} \right) \sim 1 + \frac{l_1(\lambda_1)}{n}, \ \text{as $n\to\infty$,}
\end{equation}
where $l_1(\lambda) = \log \Expect \exp(i\lambda A^{(2)})$, the logarithmic characteristic function of $A^{(2)}$.

As in \sectn{CatandMouseModel}, we reformulate the distribution of r.v. $\tau(1)$ using the time needed for the simple random walk to hit a fixed point.  Let $\{\tau^{(1)}_k\}_{k=1}^\infty$ be independent copies of $\tau$, the time needed  for the simple random walk to hit $0$ if it starts from $1$, independent of $\{\xi^{(2)}_n\}_{n=1}^\infty$. Given $\xi^{(2)}_1 = m \neq 0$, r.v. $\tau(1)$ is the time needed for the random walk $S^{(1)}_n$ to reach $m$ from zero, and thus, $\tau(1)  \overset{d}{=}  \sum_{k=1}^{|m| } \tau^{(1)}_k$. Given $\xi^{(2)}_1 = 0$, we have $\tau(1)  \overset{d}{=} 1 + \tau^{(1)}_1$. Then we have the following relation for $\tau(1)$:
\begin{align*}
\tau(1) \overset{d}{=} I[\xi^{(2)}_1 \neq 0] \sum_{k=1}^{|\xi^{(2)}_1| } \tau^{(1)}_k + I[\xi^{(2)}_1 = 0](1 + \tau^{(1)}_1).
\end{align*}
Since $\PB\{\tau^{(1)}_1 > n\} \sim \sqrt{2/(\pi n)}$, as $n\to\infty$, we conclude from Proposition $\ref{Wald's identity case for infinite mean}$ (see Appendix) that 
\begin{align*}
\PB\{\tau(1) > n\} \sim (\Expect|\xi^{(2)}_1| + \PB\{\xi^{(2)}_1 = 0\}) \PB\{\tau > n\}.
\end{align*}
Thus, there exists a r.v. $D^{(2)}$ having a stable distribution with index $1/2$ such that
\begin{align*}
\frac{\tau(n)}{n^2}  \Rightarrow D^{(2)}, \ \text{as $n\to\infty$}.
\end{align*}

Using the same argument as for \eq{StepCharFunction} we get
\begin{equation}\label{eq:WaitTimeCharFunction}
\Expect \exp\left(i\lambda_2\frac{\tau}{n^2} \right) \sim 1 + \frac{l_2(\lambda_2)}{n}, \ \text{as $n\to\infty$,}
\end{equation}
where $\lambda_2(\lambda) = \log \Expect \exp(i\lambda D^{(2)}/ (\Expect|\xi^{(2)}_1| + \PB\{\xi^{(2)}_1 = 0\}))$, the logarithmic characteristic function of $D^{(2)}/ (\Expect|\xi^{(2)}_1| + \PB\{\xi^{(2)}_1 = 0\})$.
We proceed as follows:
\begin{align*}
\Expect \exp &\left( i\left[\lambda_1\frac{\xi^{(2)}_1}{b(n)} + \lambda_2\frac{ \tau(1)}{n^2}\right]\right)\\
& =\sum_{-\infty}^{\infty} \exp\left(i\lambda_1\frac{k}{b(n)} \right)\PB\{\xi^{(2)}_1 = k\}\Expect \exp\left( i\lambda_2\frac{\tau(1)}{n^2} | \xi^{(2)}_1 = k \right) \\
& = \PB\{\xi^{(2)}_1 = 0\}\Expect \exp\left( i\lambda_2 \frac{1 + \tau}{n^2}\right) +\\
& + \sum_{k\neq 0} \exp\left(i\lambda_1\frac{k}{b(n)} \right)\PB\{\xi^{(2)}_1 = k\}\left(\Expect \exp\left( i\lambda_2\frac{\tau}{n^2} \right)\right)^{|k|}.
\end{align*}
In order to transform the last sum in the last equation, we use \eq{WaitTimeCharFunction} and get that, uniformly in $m>0$,
\begin{align*}\label{CM_simple_unform_bound}
\left(\Expect \exp\left( i\lambda_2\frac{\tau}{n^2} \right)\right)^m  & = \left( 1 + \frac{l_2(\lambda_2)}{n} + o\left(\frac{1}{n}\right)\right)^m \notag \\
& = \exp\left( m\ln(1 + \frac{l_2(\lambda_2)+o(1)}{n})\right) \notag \\
& = \exp\left( \frac{m l_2(\lambda_2)}{n}(1 + o(1)) \right) \notag \\
& = 1 + \frac{m l_2(\lambda_2)(1 + o(1))}{n} +  \frac{1}{n^2}\sum_{j=2}^\infty \frac{(m l_2(\lambda_2)(1 + o(1)))^j}{n^{j-2} j!},
\end{align*}
as $n\to\infty$. Since the latter equation is uniform in $m>0$, we get
\begin{align*}
\sum_{k\neq 0} & \exp\left(i\lambda_1\frac{k}{b(n)} \right)\PB\{\xi^{(2)}_1 = k\}\left(\Expect \exp\left( i\lambda_2\frac{\tau}{n^2} \right)\right)^{|k|}\\  & = \left(\Expect\exp\left(i\lambda_1\frac{\xi^{(2)}_1}{b(n)} \right) -\PB\{\xi^{(2)}_1 = 0\}\right)\\
& + \frac{l_2(\lambda_2)}{n} \sum_{k\neq 0} |k| \exp\left(i\lambda_1\frac{k}{b(n)} \right)\PB\{\xi^{(2)}_1 = k\} + \frac{1}{n^2}\sum_{k\neq 0}\sum_{j=2}^\infty \frac{(k l_2(\lambda_2)(1 + o(1)))^j}{n^{j-2} j!}+o\left(\frac{1}{n}\right),
\end{align*}
as $n\to\infty$. Now we use the fact that if $\sum_{-\infty}^{\infty} A_n = \sum_{-\infty}^{\infty} B_n + \sum_{-\infty}^{\infty} C_n$ and if series $\sum_{-\infty}^{\infty}A_n$ and $\sum_{-\infty}^{\infty}B_n$ converge, then $\sum_{-\infty}^{\infty} C_n$ converges too. Thus,
$$\frac{1}{n^2}\sum_{k\neq 0}\sum_{j=2}^\infty \frac{(k l_2(\lambda_2)(1 + o(1)))^j}{n^{j-2} j!} = o\left(\frac{1}{n}\right)$$
and 
\begin{align*}
\sum_{k\neq 0} & \exp\left(i\lambda_1\frac{k}{b(n)} \right)\PB\{\xi^{(2)}_1 = k\}\left(\Expect \exp\left( i\lambda_2\frac{\tau}{n^2} \right)\right)^{|k|}\\
& =\left(\Expect\exp\left(i\lambda_1\frac{\xi^{(2)}_1}{b(n)} \right) -\PB\{\xi^{(2)}_1 = 0\}\right)
  +\Expect|\xi^{(2)}_1| \frac{l_2(\lambda_2)}{n}  + o\left(\frac{1}{n}\right),
\end{align*}
as $n\to\infty$. Using \eq{StepCharFunction} and \eq{WaitTimeCharFunction}, we have
\begin{multline*}
\Expect \exp\left( i\left[\lambda_1\frac{\xi^{(2)}_1}{b(n)} + \lambda_2\frac{ \tau(1)}{n^2}\right]\right)\\
 = 1 + \frac{l_1(\lambda_1)}{n} + (\Expect|\xi^{(2)}_1| + \PB\{\xi^{(2)}_1 = 0\}) \frac{l_2(\lambda_2)}{n} + o\left(\frac{1}{n}\right),
\end{multline*}
as $n\to\infty$. We have proved that equation \eq{joint characteristic function} holds with $f_1(\lambda_1) = l_1(\lambda_1)$ and $f_2(\lambda_2) = (\Expect|\xi^{(2)}_1| + \PB\{\xi^{(2)}_1 = 0\}) l_2(\lambda_2)$. Therefore, equation \eq{asymp_indep} holds and we can use Proposition $\textbf{\ref{Jurlewicz et al}}$ to prove the first part of Theorem \textbf{\ref{thr:GeneralMouse}}.

We now turn to the second part and assume $\Expect\xi^{(2)} = \mu\neq 0$. Then the above arguments are applicable to $\sum_{k=1}^{\eta(t) + 1} (\xi^{(2)}_k - \mu)$. Thus, we have shown that the process $\left( \left(\sum_{k=1}^{\eta(nt) + 1} (\xi^{(2)}_k - \mu) \right)/ b(\sqrt{n}), \ t\ge 0 \right) $ weakly converges to the limiting one (see Appendix A for corresponding definitions). Since $\mu < \infty$, we have $b(n) = o(n)$, as $n\to\infty$. Indeed, by the Strong Law of Large Numbers,
\begin{equation*}
\frac{\sum_{k=1}^{n} (\xi^{(2)}_k - \mu)}{n} \overset{a.s.}{\to} 0, \ \text{as $n\to\infty$.}
\end{equation*}
Since $A^{(2)}$ in \eq{stable law A} is not deterministic, the denominator in the LHS of \eq{stable law A} must be an $o(n)$ function. Therefore the process
\begin{align*}
\left( \frac{\sum_{k=1}^{\eta(nt) + 1} (\xi^{(2)}_k - \mu)}{\sqrt{n}}, \ t\ge 0 \right) =\left(  \frac{\sum_{k=1}^{\eta(nt) + 1} (\xi^{(2)}_k - \mu)}{b(\sqrt{n})} \frac{b(\sqrt{n})}{\sqrt{n}} , \ t\ge 0\right)
\end{align*}
converges to the zero-valued process. Thus, it follows from the representation
\begin{align*}
\frac{\widehat{M}(nt)}{\sqrt{n}} = \frac{\sum_{k=1}^{\eta(nt)+1} (\xi^{(2)}_k - \mu)}{\sqrt{n}} + \frac{\mu (\eta(nt) +1)}{\sqrt{n}}
\end{align*}
and from the Corollary of Theorem 3.2 from Meerschaert and Scheffler (2004) that
\begin{align*}
\left\{ \frac{\widehat{M}(nt)}{\sqrt{n}}, \ t\geq 0 \right\} \overset{\mathcal{D}}{\Rightarrow} \{\mu E^{(2)}(t), \ t\geq 0\}, \ \text{as $n\to\infty$}.
\end{align*}

\textbf{5.1.2} We now present a proof of \thr{general cat and mouse}. Under the assumption of the finiteness of second moments, we can extend our result to the case where both $\xi^{(1)}$ and $\xi^{(2)}$ have general distributions. Assume now that $\{S^{(1)}_n\}_{n=0}^\infty = \{\sum_{k=1}^n \xi^{(1)}_k\}_{n=0}^\infty$ is an aperiodic random walk with zero-mean and finite-variance-$\sigma_1^2$ increments. The theory of general random walks and their hitting times is well developed. Nevertheless, uniform in terms of the hitting point results are rather scarce. It follows from Section \textbf{3.3} of Uchiyama (2011a) that, uniformly in $x$,
\begin{equation}\label{eq:general_char_func_expansion}
\Expect\left[ \exp\left( it  \tau(1)\right) | \  \xi^{(2)}_1= x\right] = 1 - (a^*(x) + e_x(t))(\sigma_1\sqrt{-2it} + o(\sqrt{|t|})), \ \text{as $t\to 0$,}
\end{equation}
where
\begin{align}
a^*(x) = 1 + \sum_{n=1}^\infty \left(\PB\{S^{(1)}_n = 0\} - \PB\{S^{(1)}_n = -x\} \right),
\end{align}
\begin{align}
e_x(t)= c_x(t) + is_x(t),
\end{align}
\begin{align}\label{eq:Important inequality from Uchiyama}
|c_x(t)| = O\left(x^2\sqrt{|t|} \right), \ \text{as $t\to 0$, uniformly in $x$,} 
\end{align}
\begin{align}\label{eq:Last Uchiyama asymptotic}
s_0(t) = 0 \ \text{and} \ \frac{s_x(t)}{x} = o(1), \ \text{as $t\to 0$, uniformly in $x\neq 0$.}
\end{align}

Following steps similar to those used in the previous part we take $t = \lambda_2 / n^2$ and, eventually, let $n$ become large. A very important relation here is \eq{Important inequality from Uchiyama}. When we take the characteristic function $\Expect \exp\left( i\left[ \lambda_1\frac{\xi^{(2)}_1}{\sigma_2\sqrt{n}} + \lambda_2 \frac{\tau(1)}{n^2}\right]\right)$ and start to separate it into different summands, the relation \eq{Important inequality from Uchiyama} leads to the summand (see details below)
\begin{align*}
\sum_{x \in \mathbb{Z}} O\left(\frac{x^2}{n^2}\right) \PB\{\xi^{(2)}_1 = x\}, \ \text{as $n\to\infty$,}
\end{align*}
and this is the main reason why we need to assume that $\xi^{(2)}_1$ has a finite second moment.

Assume now that $\Expect \xi^{(2)}_1 = 0$ and $\sigma_2 = \D \xi^{(2)}_1 < \infty$. We have (see, e.g., Proposition 7.2 from Uchiyama (2011b))
\begin{align}\label{eq:Asymptotic for a*}
\sigma^2_1(a^*(x) - I(x=0)) \sim |x|, \ \text{ as $|x|\to\infty$.}
\end{align} 
As a consequence, we get $\Expect a^*(\xi^{(2)}_1) < \infty$. Let $p^{(2)}(x) = \PB\{\xi^{(2)}_1 = x\}$. Then the total probability formula gives us

\begin{multline*}
\Expect \exp\left( i\left[ \lambda_1\frac{\xi^{(2)}_1}{\sigma_2\sqrt{n}} + \lambda_2 \frac{\tau(1)}{n^2}\right]\right)\\
 = \sum_{x\in\mathbb{Z}}\exp\left( i\lambda_1 \left[\frac{x}{\sigma_2\sqrt{n}} \right] \right)\Expect\left[ \exp\left( i\left[\lambda_2 \frac{\tau(1)}{n^2}\right]\right) | \  \xi^{(2)}_1= x\right]  p^{(2)}(x).
\end{multline*}
Now we use \eq{general_char_func_expansion}-\eq{Last Uchiyama asymptotic} to get
\begin{align*}
\Expect \exp\left( i\left[ \lambda_1\frac{\xi^{(2)}_1}{\sigma_2\sqrt{n}} + \lambda_2 \frac{\tau(1)}{n^2}\right]\right) & =  \Expect \left[\exp\left(  i\lambda_1 \left[\frac{\xi^{(2)}_1}{\sigma_2\sqrt{n}} \right]\right)\right] -\\
 &  - \frac{\sigma_1 \sqrt{-2i\lambda_2}}{n} \Expect \left[a^*(\xi^{(2)}_1) \exp\left(  i\lambda_1 \left[\frac{\xi^{(2)}_1}{\sigma_2\sqrt{n}} \right]\right)\right] +\\
 & + O\left( \frac{1}{n^2} \Expect \left[\left( \xi^{(2)}_1\right)^2  \exp\left(  i\lambda_1 \left[\frac{\xi^{(2)}_1}{\sigma_2\sqrt{n}} \right]\right)\right]\right)  +\\
 & + o\left( \frac{1}{n} \Expect \left[\xi^{(2)}_1  \exp\left(  i\lambda_1 \left[\frac{\xi^{(2)}_1}{\sigma_2\sqrt{n}} \right]\right)\right]\right)  + o\left(\frac{1}{n} \right),
\end{align*}
as $n\to\infty$. Next, we use relation \eq{Asymptotic for a*} and the Taylor expansion for the exponent to get
\begin{align*}
 \Expect \left[a^*(\xi^{(2)}_1) \exp\left(  i\lambda_1 \left[\frac{\xi^{(2)}_1}{\sigma_2\sqrt{n}} \right]\right)\right] =  \Expect \left[a^*(\xi^{(2)}_1) \right] + o(1), \ \text{as $n\to\infty$.}
\end{align*}
Since $\Expect \xi^{(2)}_1 = 0$ and $\D  \xi^{(2)}_1 < \infty$, the Central Limit Theorem holds. Thus, we have the analogue of \eq{StepCharFunction} with $l_1$ being a logarithmic characteristic function of a r.v. with a standard normal distribution. Finally, we get
\begin{multline*}
\Expect \exp\left( i\left[ \lambda_1\frac{\xi^{(2)}_1}{\sigma_2\sqrt{n}} + \lambda_2 \frac{\tau(1)}{n^2}\right]\right) = 1 +  \frac{l_1(\lambda_1)}{n} - \frac{\sigma_1 \sqrt{-2i\lambda_2}}{n} \Expect \left[a^*(\xi^{(2)}_1) \right]  + o\left(\frac{1}{n}\right).
\end{multline*}
Thus, we proved equation \eq{joint characteristic function} for this case and the rest of the proof follows the same argument as in the previous case.

\subsection{Proof of \thr{DogCatMouse}}\label{sect:dog cat mouse}

Recall that random vectors $\{Y^{(3)}_n, J^{(3)}_n\}_{n=1}^\infty$ are i.i.d., where $Y^{(3)}_1 = \sum_{k=1}^\nu \gamma^{(3)}_k$ and $J^{(3)}_1 = T^{(3)}(1) =  U^{(3)}(\nu)$. We have
\begin{align*}
N(t) = \max\{n> 0: \ \sum_{k=1}^n J^{(3)}_k \leq t\} \ \text{and} \ \widetilde{M}(t) = \sum_{k = 1}^{N(t)} Y^{(3)}_k.
\end{align*}
From Propositions \textbf{\ref{tail of J}} and \textbf{\ref{moments of Y_m}} we have
\begin{align*}
\Expect Y^{(3)}_1 = 0, \ \D Y^{(3)}_1 = 2, \ \Expect (Y^{(3)}_1)^m < \infty, \ \text{for $m\geq 2$,} \ \text{and} \ \PB\{J^{(3)}_1 > n\} \sim \frac{2^{7/4}}{\Gamma(3/4) n^{1/4}},
\end{align*}
as $n\to\infty$. From Theorem 5.1 from Kasahara (1984) we have
\begin{equation}\label{eq:DCM_result for M_tilde}
\left\lbrace \frac{\widetilde{M}(nt)}{2^{-7/8}n^{1/8}\sqrt{\D Y^{(3)}_1}}, t\geq 0\right\rbrace \overset{\mathcal{D}}{\Rightarrow} \{B(E^{(3)}(t)), t\geq 0\}, \ \text{as $n\to\infty$,}
\end{equation}
where $B(t)$ is a standard Brownian motion, independent of $E^{(3)}(t)$.

We show now that \eq{DCM_result for M_tilde} holds with $M(nt)$ in the place of $\widetilde{M}(nt)$. It is sufficient to prove that, for any fixed $T>0$,
\begin{align*}
\frac{\max_{1 \leq k \leq [nT]}\left\lbrace \widetilde{M}_{k} - M_{k}\right\rbrace}{n^{1/8}} \overset{a.s.}{\to} 0, \ \text{as $n\to\infty$. }
\end{align*}

For $k=1,2,\ldots$, let $\nu_k$ be the number of visits of the set $\{(0,0),(0,1)\}$ by Markov chain $V_n$ within time interval $(T_{k-1}^{(3)}, T_k^{(3)}]$, and let $\varkappa_k$ be the total number of jumps of the mouse within this time interval. Further, let
$R_k = \max_{T^{(3)}(k-1) \leq l \leq T^{(3)}(k)} |\widetilde{M}_l - M_l|$. The triples 
$(\nu_k,\varkappa_k,R_k)$ are i.i.d., 
 the pairs $(\nu_k,\varkappa_k)$ are
i.i.d. copies of the pair $(\nu,\varkappa)$ introduced earlier,
and
\begin{align}\label{ququ}
R_k \le \varkappa_k \le 2\nu_k \ \ \mbox{a.s.}
\end{align}

Further, all random variables in \eqref{ququ} have a finite exponential moment, ${\mathbb E} \exp (cR_1)<\infty$ for some $c>0$. 

For $K>0$, let $\widehat{J}_j = \min (J^{(3)}_j, K)$.
Since ${\mathbb E} J^{(3)}_1=\infty$, one can choose $K$ such that $A:={\mathbb E} \widehat{J}_1> T$.
Let $\widehat{S}_n=\sum_{i=1}^n \widehat{J}_i$. Then, for any $C>0$,
\begin{align*}
{\mathbb P} (\eta (nT)+1 >n) &\le {\mathbb P}
(\widehat{S}_n\le nT) \le \left(e^{CT}{\mathbb E} e^{-C\widehat{J}_1} \right)^n,
\end{align*}
where the term in the parentheses on the right-hand side may be made less than $1$ for $C>0$ sufficiently small. Next,
\begin{align*}
{\mathbb P} 
( \max_{1\le k \le n} R_k \ge n^{1/8} \varepsilon )
&\le n
{\mathbb P} (R_1  \ge n^{1/8} \varepsilon )\le n {\mathbb E} e^{cR_1} \cdot e^{-c\varepsilon n^{1/8}}.
\end{align*}
Then, for any $\varepsilon >0$, 
\begin{align*}
{\mathbb P} \left(\frac{\max_{1 \leq k \leq [nT]}\left\lbrace \widetilde{M}_{k} - M_{k}\right\rbrace}{n^{1/8}} 
\ge \varepsilon\right) 
\le 
{\mathbb P} 
( \max_{1\le k \le n} R_k \ge n^{1/8}\varepsilon ) +
{\mathbb P} (\eta (nT)\ge n),
\end{align*}
where both terms on the right-hand side are summable in $n$.
Therefore, by the 0-1 law and by arbitrariness of
$\varepsilon >0$, \eqref{ququ} follows. This completes the proof of Theorem \textbf{\ref{thr:DogCatMouse}}.

\subsection{Proof of \thr{LinearFoodChain}}\label{sect:linear food chain}

Random process $X^{(1)}$ is a simple random walk on $\mathbb{Z}$ and for $j \in [2, N]$ we have
\begin{multline*}
\PB\{X^{(j)}(n) - X^{(j)}(n-1) = 1| \ X^{(j)}(n-1) = X^{(j-1)}(n-1)\}\\ = \PB\{X^{(j)}(n) - X^{(j)}(n-1) = - 1| \ X^{(j)}(n-1) = X^{(j-1)}(n-1)\} = \frac{1}{2}.
\end{multline*}

Let us give a new representation for such process. Let $X^{(1)}(0) = X^{(2)}(0) = \ldots = X^{(N)}(0) = 0$ and let r.v. $T_j(n)$ denote the time when $X^{(j)}$ makes $n$-th step. Let $T_j(0) = 0$. Note the difference between $T_j$ and $T^{(3)}$. Thus, $\{X^{(j)}(T_j(k))\}_{k=0}^\infty$ is a simple random walk on $\mathbb{Z}$ and if $X^{(j)}(n) \neq X^{(j)}(n-1)$ then $n\in \{T_j(k)\}_{k=1}^\infty$. Let 
$$\xi^{(j)}_k = X^{(j)}(T_j(k)) - X^{(j)}(T_j(k-1)) =X^{(j)}(T_j(k)) - X^{(j)}(T_j(k)-1)$$
for $j\geq 1$ and $k\geq 1$. By definition $\{\{\xi^{(j)}_k\}_{k=0}^\infty\}_{j=1}^N$ are mutually independent and equal $\pm 1$ w.p. $1/2$.

Since $X^{(1)}$ jumps every time, $T_1(k) = k$ for $k\geq 0$. Let $\tau$ be the time that the simple random walk goes from point $1$ to $0$. From \sectn{CatandMouseModel} it is easy to deduce that the time between meeting-time instants of the cat and the mouse has the same distribution as $\tau$. Thus, if we look at the system only at the times $\{T_j(k)\}_{k=0}^\infty$ the time between meeting time-instants of $X^{(j)}$ and $X^{(j+1)}$ has the same distribution as $\tau$. 

Let us define $\nu_j(n) = \max\{k\ge 0: \ T_j(k) \le n\}$, the number of time-instants up to time $n$ when $X^{(j)}$ changed its value. Then we can rewrite the dynamics of the $j$-th coordinate as
$$X^{(j)}(n) = \sum_{k=1}^{\nu_j(n)} \xi^{(j)}_{T_j(k)}.$$
Our assumptions on the distribution of the increments $\xi^{(j)}_k$, for $k\ge 1$, provide us with the next important property of our process.

\begin{proposition}
Sequences $\{T_j(k)\}_{k=0}^\infty$ and $\{\xi^{(j)}_k\}_{k=1}^\infty$ are independent for any $j\in\{1, \ldots, N\}$.
\end{proposition}

This property comes from the space-symmetry of the model. Indeed, for $j=1$ the result is trivial, since $\nu_1(n) = n$. We show the result for $j=2$ and then extend it onto $j>2$.
Define
\begin{align*}
{}^1\tau(0) = 0 \ \text{and} \ {}^1\tau(k) = \inf\{n > {}^1\tau(k-1): \ X^{(1)}(n) = X^{(2)}(n)\}, \ \text{for $k\ge 1$}.
\end{align*}
One can see that in our model $T_2(k) = 1 + {}^1\tau(k-1)$, for $k\ge 1$. In the time interval $[1, {}^1\tau(1)]$ the second coordinate changes its value only at the time $T_2(1) = 1$. Thus, the time ${}^1\tau(1)$ does not depend on $\xi^{(2)}_k$, for $k\ge 2$. Additionally, the trajectory $\{X^{(1)}(n)\}_{n=0}^\infty$ has the same distribution as $\{-X^{(1)}(n)\}_{n=0}^\infty$. Thus, 
\begin{align*}
\PB\{{}^1\tau(1) = n, \ \xi^{(2)}_1 = 1\} = \PB\{{}^1\tau(1) = n, \ \xi^{(2)}_1 = -1\}.
\end{align*}
As a corollary of the last equation, we get that ${}^1\tau(1)$ has the same distribution as the time that is needed  for the simple random walk to hit $0$ if it starts from $1$. This implies that
\begin{align}\label{tail_of_standard_time}
\PB\{{}^1\tau(1) > n\} \sim \sqrt{\frac{2}{\pi n}}, \ \text{as $n\to\infty$.}
\end{align}

From the symmetry of our model, it follows further that the sequence $\{{}^1\tau(k)\}_{k=0}^\infty$, and subsequently the sequences $\{T_2(k)\}_{k=0}^\infty$ and $\{\nu_2(n)\}_{n=1}^\infty$, do not depend on $\{\xi^{(2)}_k\}_{k=1}^\infty$ (and, therefore, on $\{\xi^{(j)}_k\}_{k\ge 1, j\ge2}$).

For the analysis of $\{T_j(k)\}_{k=0}^\infty$, $j > 2$, we need to define an 'embedded version' of ${}^1\tau(k)$. Let 
\begin{align*}
{}^j\tau(0) = 0 \ \text{and} \ {}^j\tau(k) = \inf\{m > {}^j\tau(k-1): \ X^{(j)}(T_j(m)) = X^{(j+1)}(T_j(m))\}, \ \text{for $k\ge 1$}.
\end{align*}
The process $\{{}^j\tau(k)\}_{k=0}^\infty$ counts the number of times that the process $X^{(j)}$ changed its value between the time-instants when $X^{(j)}$ and $X^{(j+1)}$ have the same value. Using the same argument as before, we get that the sequence $\{{}^j\tau(k)\}_{k=0}^\infty$ does not depend on $\{\xi^{(j+1)}_k\}_{k=1}^\infty$.

The $j$-th coordinate $X^{(j)}$ changes its value for the $k$-th time at a time-instant $n$ if and only if up to time $n-1$ processes $X^{(j-1)}$ and $X^{(j)}$ had the same value exactly $k-1$ times (not including $X^{(j-1)}(0) = X^{(j)}(0) = 0$) and the last time was at the time-instant $n-1$ (which also means that at the  time-instant $n-1$ the process $X^{(j-1)}$ changes its value). This can be rewritten as
\begin{align*}
T_j(k) = n \ \Leftrightarrow \ n-1 = T_{j-1}({}^{j-1}\tau(k-1)), \ \text{for $j\ge 2$, $k\ge 1$,}
\end{align*}
and thus $T_j(k) = 1 + T_{j-1}({}^{j-1}\tau(k-1))$. Thus, since sequences $\{T_2(k)\}_{k=0}^\infty$ and $\{{}^2\tau(k)\}_{k=0}^\infty$ do not depend on $\{\xi^{(j)}_k\}_{k\ge 1, j\ge 3}$, the same holds for $\{T_3(k)\}_{k=0}^\infty$. Therefore, using the induction, we get that the sequences $\{T_j(k)\}_{k=0}^\infty$ and $\{\xi^{(j)}_k\}_{k=1}^\infty$ are independent for any $j\ge 1$.

As a corollary of this result we get
\begin{align*}
X^{(j)}(n) = \sum_{k=1}^{\nu_j(n)} \xi^{(j)}_{T_j(k)} \overset{d}{=} \sum_{k=1}^{\nu_j(n)} \xi^{(j)}_k.
\end{align*}

Let ${}^j\eta(n) = \max\{k\geq 0 :\ {}^j\tau(k) \leq n\}$ for $n\geq 0$ and $j \in [1, \ldots,  N]$. Since the sequence $\{{}^j\tau(k)\}_{k=0}^\infty$ depends only on the sequence $\{\xi^{(j)}_k\}_{k=1}^\infty$, we have that $\{{}^j\eta(n)\}_{j=1}^{N-1}$ are i.i.d. r.v.'s.  For $n\geq 1$ and $j \in \{1, \ldots, N\}$ we have
\begin{align*}
\nu_j(n) = \max\{k\geq 0:\ T_j(k) \leq n\} & = \max\{k\geq 1:\ 1 + T_{j-1}({}^{j-1}\tau(k-1)) \leq n\}\notag\\
 & =  1 + \max\{k\geq 0:\ T_{j-1}({}^{j-1}\tau(k)) \leq n - 1\}\notag\\
 & =  1 + \max\{k\geq 0:\ {}^{j-1}\tau(k) \leq \nu_{j-1}(n - 1)\}\notag\\
 & =  1 + {}^{j-1}\eta(\nu_{j-1}(n-1)).
\end{align*}

For $n < N-1$ we can iterate the process and get
\begin{align*}
\nu_N(n) & =  1 + {}^{N-1}\eta(1 + {}^{N-2}\eta(\ldots (1 + {}^{N-n}\eta(0))\ldots))\notag\\
& =  1 + {}^{N-1}\eta(1 + {}^{N-2}\eta(\ldots (1 + {}^{N-n+1}\eta(1))\ldots))\notag\\
& \overset{d}{=}  1 + {}^{n-1}\eta(1 + {}^{n-2}\eta(\ldots (1 + {}^{1}\eta(1))\ldots))\notag\\
& =  \nu_n(n).
\end{align*}
For $n \ge N-1$ we have
\begin{align*}
\nu_N(n) = 1 + {}^{N-1}\eta(1 + {}^{N-2}\eta(\ldots + {}^1\eta(n-N+1))).
\end{align*}
We want to construct a process with the same distribution as $\{\nu_N(n)\}_{n=0}^\infty$ in a form of $\nu_{N-1}(\varphi(n))$, where process $\{\varphi(n)\}_{n=0}^\infty$ is independent of everything else. Define process $\{\eta(n)\}_{n=0}^\infty \overset{d}{=} \{{}^{N-1}\eta(n)\}_{n=0}^\infty$, which is independent of everything else. Then, for $n \ge N-1$, we have
\begin{align*}
\nu_N(n) \overset{d}{=} 1 + {}^{N-2}\eta(1 + {}^{N-3}\eta(\ldots + \eta(n-N+1))).
\end{align*}
Using the same formula for $\nu_{N-1}(m)$ with such $m$ that $m- (N-1) + 1 = 1 + {}^{N-1}\eta(n-N+1)$, we get
\begin{align*}
\nu_N(n) \overset{d}{=} \nu_{N-1}(N-1 +  \eta(n-N+1)), \ \text{for $n \ge N-1$.}
\end{align*}

Then, for $n\geq N$ we have $X^{(N)}(n) \overset{d}{=} X^{(N-1)}(N -1 + \eta(n-N+1))$. There exists a non-degenerate r.v. $\zeta$ (see Section XI.5 in Feller (1971b)) such that $\PB\{\zeta_i \geq y\} = G_{1/2}(9/y^2)$ and
\begin{align*}
\eta(n)\PB\{{}^{N-1}\tau(1) > n\} \Rightarrow \zeta, \text{as $n\to\infty$.}
\end{align*}
Therefore, using $(\ref{tail_of_standard_time})$ we get
\begin{align}\label{weak convergence of eta(n-N+1)}
\frac{j -1 + \eta(n-j+1)}{\sqrt{n}} = \frac{j -1 + \eta(n-j+1)}{\sqrt{n-j+1}}\frac{\sqrt{n-j+1}}{\sqrt{n}} \Rightarrow \sqrt{\frac{\pi}{2}}\zeta,
\end{align}
as $n\to\infty$, for $j\ge 1$. We now present a known result that we utilise to prove Theorem \textbf{\ref{thr:LinearFoodChain}}.
\begin{proposition}\label{Dobrushin_lemma}
\textit{(Dobrushin (1955), (v))} Let $Y(t)$ and $\tau_n$ be independent sequences of r.v.'s such that
\begin{align}\label{Dobrushin_condition}
\frac{Y(t)}{bt^\beta} \Rightarrow Y, \ \text{as $t\to\infty$, \ and} \ \ \frac{\tau_n}{d n^\delta} \Rightarrow \tau, \ \text{as $n\to\infty$.}
\end{align}
Then for independent $Y$ and $\tau$ we have
\begin{align*}
\frac{Y(\tau_n)}{b d^\beta n^{\beta\delta}} \Rightarrow Y \tau^\beta , \ \text{as $n\to\infty$.}
\end{align*}
\end{proposition}

Indeed, by the Central Limit Theorem $X^{(1)}(n) / \sqrt{n}$ weakly converges to a normally distributed r.v. $\psi$ (we assume that $\psi$ and $\zeta$ are independent). Together with $(\ref{weak convergence of eta(n-N+1)})$ and independence of $X^{(1)}(n)$ and $\eta(n)$, this ensures that condition $(\ref{Dobrushin_condition})$ holds with $Y(t) = X^{(1)}([t])$, $\tau_n = 1 + \eta(n-1)$ and  $\beta = \delta = 1/2$. By Proposition \textbf{\ref{Dobrushin_lemma}},
\begin{align*}
\frac{X^{(2)}(n)}{n^{1/4}} \overset{d}{=} \frac{X^{(1)}(1 + \eta(n-1))}{n^{1/4}} \Rightarrow \psi \sqrt{\sqrt{\frac{\pi}{2}}\zeta}, \ \text{as $n\to\infty$.}
\end{align*}
Let $\{\zeta_j\}_{j=2}^{N}$ be independent copies of $\zeta$ which are independent of $\psi$. Next, we use the induction argument. For some $j\ge 1$ condition $(\ref{Dobrushin_condition})$ holds with $Y(t) = X^{(j)}([t])$, $\tau_n = j-1 + \eta(n-j+1)$,  $\beta = 2^{-j}$ and $\delta = 2^{-1}$. By Proposition \textbf{\ref{Dobrushin_lemma}}, we get
\begin{align*}
\frac{X^{(j+1)}(n)}{n^{2^{-(j+1)}}}\overset{d}{=} \frac{X^{(j)}(j-1 + \eta(n-j+1))}{n^{2^{-(j+1)}}} \Rightarrow \psi \prod_{i=2}^{j+1}\sqrt{\sqrt{\frac{\pi}{2}}\zeta_i}, \ \text{as $n\to\infty$.}
\end{align*}
This concludes the proof of Theorem \textbf{\ref{thr:LinearFoodChain}}.

\section*{Acknowledgements}

The authors would like to thank two anonymous referees for their helpful and constructive comments. 

\appendix
\section*{Appendix}
\section{Weak convergence for processes from $D[[0,\infty), \mathbb{R}]$}\label{app:m_1-topology}
To make the paper self-contained, we recall the definition of the $\mathcal{J}_1$-topology (see, e.g., Skorokhod (1956)).
Let $D[[0,T], \mathbb{R}]$ denote the space of all right-continuous functions on $[0, T]$ having left limits (RCLL or c\`{a}dl\`{a}g). For any $g \in D[[0,T], \mathbb{R}]$ let $\bigl\|g\bigr\| = \sup_{t\in[0, T]} |g(t)|$.

Let $\Lambda_T$ be the set of increasing continuous functions $\lambda: [0, T] \to [0, T]$, such that $\lambda(0) = 0$ and $\lambda(T) = T$. Let $\lambda_{id, T}$ denote the identity function. Then
$$d_{\mathcal{J}_1, T}(g_1, g_2) = \inf_{\lambda\in\Lambda_T} \max( \bigl\|g_1 \circ \lambda - g_2\bigr\|,  \bigl\| \lambda - \lambda_{id, T}\bigr\|)$$
defines a metric inducing $\mathcal{J}_1$.

Let $D[[0,\infty), \mathbb{R}]$ be the space of all RCLL functions on the positive half-line. On the space $D[[0, \infty), \mathbb{R}]$ the $\mathcal{J}_1$-topology is defined by the metric 
$$d_{\mathcal{J}_1, \infty}(g_{1}, g_{2}) = \int_0^\infty e^{-t} \min(1, d_{\mathcal{J}_1, T}(g_{1,T}, g_{2,T})) dT,$$
where, for $i=1,2$,  $g_{i,T}$ is the restriction of function $g_i$ on the interval $[0,T]$. 
 
Convergence $g_n \to g$ in $(D[[0, \infty), \mathbb{R}], \tau)$ means that $d_{\tau, T}(g_n, g) \to 0$ for every continuity point $T$ of $g$ (see, Whitt (2002)).

Let $\{\{X_n(t)\}_{t\geq 0}\}_{n=1}^\infty$ and $\{X(t)\}_{t\geq 0}$ be stochastic processes with trajectories from $D[[0,\infty), \mathbb{R}]$. We say that weak convergence 
$$\{X_n(t)\}_{t\geq 0} \overset{\mathcal{D}}{\Rightarrow} \{X(t)\}_{t\geq 0},$$
holds if
$$\Expect f(\{X_n(t)\}_{t\geq 0}) \to \Expect f(\{X(t)\}_{t\geq 0}), \ \text{as $n\to\infty$,}$$
for any continuous and bounded function $f$ on $D[[0,\infty), \mathbb{R}]$ endowed with the $\mathcal{J}_1$-topology.

\begin{proposition}\label{auxiliary statement for equivalent convergence in J_1}
Let $\{X_n\}_{n=1}^\infty$ and $\{Y_n\}_{n=1}^\infty$ be two sequences of stochastic processes  with trajectories from $D[[0, \infty), \mathbb{R}]$. Given $d_{\mathcal{J}_1, \infty}(X_n, Y_n) \overset{a.s.}{\to} 0$, we have
$$X_n - Y_n \overset{\mathcal{D}}{\Rightarrow} 0.$$ 
\end{proposition}

\section{Asymptotic closeness of two scaled processes}\label{app:NewProp3}
In this section we prove the Remark to Theorem $1$. We consider a general CM MC $(C_n, M_n)$, $n \ge 0$, and the jumps $\xi^{(2)}_k$ of the second component are bounded r.v.'s. Then for the process $M(nt) = M_{[nt]}$ we have the same functional limit theorem as for the process $\widehat{M}(nt) = M(nt +1)$. Due to Proposition \ref{auxiliary statement for equivalent convergence in J_1}, it is sufficient to prove the following.
\begin{proposition}\label{auxiliry lemma about convergence in J_1}
We have
\begin{align*}
d_{\mathcal{J}_1, \infty}\left(\left\lbrace \frac{M(nt)}{b(\sqrt{n})}, t\geq 0 \right\rbrace,\left\lbrace \frac{M(nt + 1)}{b(\sqrt{n})}, t\geq 0 \right\rbrace\right) \overset{a.s.}{\to} 0, \ \text{as $n\to\infty$. }
\end{align*}
\end{proposition}
\begin{proof}
First, we restrict our processes to the time interval $[0, T]$, with an arbitrary finite $T$, and investigate the convergence of the distance $d_{\mathcal{J}_1, T} (\cdot, \cdot)$ between our processes. Second, we bound the distance using the following function $\lambda_n$:
$$\lambda_n(t) = \begin{cases}
0, \ t\in [0, \frac{1}{n}],\\
t-\frac{1}{n}, \ t\in [\frac{1}{n}, t_n],\\
t_n + (t-t_n)\frac{T-t_n + 1/n}{T-t_n},
\end{cases} \ \text{where} \ t_n = \frac{[nT] - 1}{n}.$$

Thus, $M(nt) = M(n\lambda_n(t)+1)$ for $t\in [1/n, t_n]$.  Then the distance between processes on time interval $[0,T]$ can be bounded as follows:
\begin{multline*}
d_{\mathcal{J}_1, T}\left(\left\lbrace \frac{M(nt)}{b(\sqrt{n})}, t \in[0, T] \right\rbrace,\left\lbrace \frac{M(nt + 1)}{b(\sqrt{n})}, t \in[0, T] \right\rbrace\right) \\
\le  \max\left( \frac{\max(|\xi^{(2)}_1|, |\xi^{(2)}_{\eta([nT]-2) +1}|, |\xi^{(2)}_{\eta([nT]-2) +1} + \xi^{(2)}_{\eta([nT]-1) +1}|)}{b(\sqrt{n})}, \frac{1}{n} \right)\\
\le \max\left( \frac{\max(|\xi^{(2)}_1|, |\xi^{(2)}_{\eta([nT]-2) +1}|, |\xi^{(2)}_{\eta([nT]-2) +1}| + |\xi^{(2)}_{\eta([nT]-2) +2}|)}{b(\sqrt{n})}, \frac{1}{n} \right).
\end{multline*}
Since $\xi^{(2)}_k$ are bounded and $b(n) \to \infty$, as $n\to\infty$, the right-hand side of last inequality converges to zero a.s.. 
\end{proof}

\section{Tail asymptotics for randomly stopped sum}\label{app:RandStopSumAsymp}

Let $\xi_1, \xi_2, \ldots$ be positive i.i.d. r.v.'s with a common distribution function $F$. Let $S_0 = 0$ and $S_k = \xi_1 + \ldots \xi_k$, $k\geq 1$. Let $\tau$ be a counting r.v. with a distribution function $G$, independent of $\{\xi_k\}_{k=1}^\infty$. For a general overview concerning tail asymptotics of $S_\tau$ see, e.g., Denisov \textit{et al.} (2010) and references therein.  The next result follows from Theorem $1$ from Korshunov (2009).

\begin{proposition}\label{Wald's identity case for infinite mean}
Assume that $\overline{F}(x) \sim l_1(x) /x^\alpha$, $\alpha \in [0, 1)$ and $\tau$ has any distribution with $\Expect\tau < \infty$. Then
\begin{align*}
\PB\{S_\tau > n\} \sim \Expect \tau \PB\{\xi > n\} \ \text{as $n\to\infty$}.
\end{align*}
\end{proposition}

The next result we use in \lemt{DCM trajectory} and we prove it using Tauberian theorems.

\begin{proposition}\label{Proposition_S_tau}
Assume that $\overline{F}(x) \sim l_1(x) /x^\alpha$ and $\overline{G}(x) \sim l_2(x) /x^\beta$,  $\alpha, \beta \in (0,1)$. Then
\begin{align*}
\PB\{S_\tau > n\} \sim n^{-\alpha\beta} \frac{\Gamma^\beta(1-\alpha) \Gamma(1-\beta)}{\Gamma(1-\alpha\beta)}l_1^\beta\left(n \right) l_2 \left(\frac{n^\alpha}{\Gamma(1-\alpha)l_1\left(n\right)} \right), \ \text{as $n\to\infty$.}
\end{align*}
\end{proposition}

\begin{proof}
Denote the c.d.f. of $S_\tau$ by $H$. Let 
\begin{align*}
\overline{F}(x) = 1- F(x), \ x\in \mathbb{R},
\end{align*}
\begin{align*}
\widehat{F}(\lambda) = \Expect e^{-\lambda \xi_1} = \int_0^\infty e^{-\lambda x} dF(x), \ \lambda \ge 0.
\end{align*}

Define $\overline{G}, \widehat{G}, \overline{H}$, and $\widehat{H}$ similarly.  We use the following result.

\begin{proposition}
\textbf{(part of Corollary 8.1.7, Bingham, Goldie and Teugels (1987))} For a constant $ \alpha \in [0, 1]$, and for a slowly varying at infinity function $l$, the following are equivalent:
\begin{align*}
1 - \widehat{F}(\lambda) \sim \lambda^\alpha l\left(\frac{1}{\lambda} \right), \ \text{as $\lambda \downarrow 0$, } 
\overline{F}(x) \sim \frac{l(x)}{x^\alpha \Gamma(1-\alpha)}, \ \text{as $x\to \infty$,} & \text{if $0\le\alpha < 1.$}
\end{align*}
\end{proposition} 

Using this result, we get
\begin{align*}
1-\widehat{F}(\lambda) \sim \lambda^\alpha \Gamma(1-\alpha)l_1\left(\frac{1}{\lambda} \right) \ \text{and} \ 1-\widehat{G}(\lambda) \sim \lambda^\beta \Gamma(1-\beta)l_2\left(\frac{1}{\lambda} \right), \ \text{as $\lambda\downarrow 0$.}
\end{align*}

Let us analyse $\widehat{H}$:
\begin{align*}
\widehat{H}(\lambda) = \Expect e^{-\lambda S_\tau} = \sum_{k=1}^\infty e^{-\lambda (\xi_1 + \ldots + \xi_k)} \PB\{\tau = k\} = \Expect \left(\Expect e^{-\lambda \xi_1} \right)^\tau  = \widehat{G}(-\ln \widehat{F}(\lambda)).
\end{align*}

Since
\begin{align*}
-\ln \widehat{F}(\lambda) = -\ln(1 - (1- \widehat{F}(\lambda))) \sim 1- \widehat{F}(\lambda), \ \text{as $\lambda\downarrow 0$,}
\end{align*}

we have
\begin{multline}\label{eq:Tauberian_equivalence}
1 - \widehat{H}(\lambda) \sim 1 - \widehat{G}\left(\lambda^\alpha \Gamma(1-\alpha)l_1\left(\frac{1}{\lambda} \right) \right)\\
 \sim \lambda^{\alpha\beta}\Gamma^\beta(1-\alpha) \Gamma(1-\beta) l_1^\beta\left(\frac{1}{\lambda} \right) l_2 \left(\frac{1}{\lambda^\alpha \Gamma(1-\alpha)l_1\left(\frac{1}{\lambda} \right)} \right),
\end{multline}
as $\lambda \downarrow 0$, and finally
\begin{align*}
\overline{H}(x) \sim x^{-\alpha\beta} \frac{\Gamma^\beta(1-\alpha) \Gamma(1-\beta)}{\Gamma(1-\alpha\beta)}l_1^\beta\left(x \right) l_2 \left(\frac{x^\alpha}{\Gamma(1-\alpha)l_1\left(x\right)} \right), \ \text{as $x\to\infty$.}
\end{align*}
Note that the function $l_2(h(\lambda))$ with $h(\lambda) = 1/ (\lambda^\alpha \Gamma(1-\alpha)l_1\left(\frac{1}{\lambda} \right))$ on the right-hand side of \eq{Tauberian_equivalence} is a slowly varying at infinity function. Indeed, for any constant $c \neq 0$, as $\lambda \downarrow 0$,
$$h(c\lambda) = \frac{1}{(c\lambda)^\alpha \Gamma(1-\alpha)l_1\left(\frac{1}{c\lambda} \right)} \sim c^{-\alpha}h(\lambda) $$
and therefore $l_2(h(c\lambda)) \sim l_2(c^{-\alpha} h(\lambda)) \sim l_2(h(\lambda))$.
\end{proof}

\end{document}